\documentclass[a4paper,12pt]{amsart}

\usepackage{fullpage}

\makeindex

\usepackage{xcolor}
\usepackage[
colorlinks=true,
citecolor=blue,
urlcolor=blue,
linkcolor=blue,
pdfborder={0 0 0},
breaklinks
]{hyperref}

\usepackage{pifont}
\usepackage{amsfonts,graphics,amsmath,amsthm,amscd,amssymb,amsmath,latexsym,euscript, enumerate,kotex,mathtools}
\usepackage{epsfig,url}
\usepackage{flafter}
\usepackage[all,cmtip,line]{xy}
\usepackage{array}
\usepackage[english]{babel}
\usepackage{overpic}
\usepackage{subfig}
\usepackage{multirow}
\usepackage{tikz-cd}
\usepackage[T1]{fontenc}
\usepackage{lmodern}
\usepackage{pifont}

\usepackage{microtype}
\usepackage{wrapfig}
\usepackage[shortlabels]{enumitem}
\setlist[enumerate, 1]{1\textsuperscript{o}}

\newtheorem{theorem}{Theorem}[section]
\newtheorem{lemma}[theorem]{Lemma}

\newtheorem{question}[theorem]{Question}
\newtheorem{corollary}[theorem]{Corollary}

\theoremstyle{definition} 
\newtheorem{definition}[theorem]{Definition}
\newtheorem{definition-lemma}[theorem]{Definition-Lemma}

\newtheorem{example}[theorem]{Example}

\theoremstyle{remark}
\newtheorem{remark}[theorem]{Remark}

\numberwithin{equation}{section}

\newcommand{\R}{\mathbb{R}}
\newcommand{\Z}{\mathbb{Z}}

\newcommand{\Q}{\mathbb{Q}}

\newcommand{\A}{\mathbb{A}}

\newcommand{\whitestar}{\mathrel{\text{\ding{73}}}}

\def\Spec{\operatorname{Spec}}

\def\Supp{\operatorname{Supp}}

\newcommand{\floor}[1]{\left\lfloor #1 \right\rfloor}
\newcommand{\ceil}[1]{\left\lceil #1 \right \rceil}
\let\oldframe\frame
\renewcommand\frame[1][allowframebreaks]{\oldframe[#1]}

\title[On the existence of minimizer on a log Fano cone singularity]{On the existence of minimizer on a log Fano cone singularity}

\date{\today}

\begin{document}

\author[D.~Kim]{Donghyeon Kim}
\address[Donghyeon Kim]{Department of Mathematics, Yonsei University, 50 Yonsei-ro, Seodaemun-gu, Seoul 03722, Republic of Korea}
\email{narimial0@gmail.com}

\begin{abstract}
We prove that if $x\in (X,\Delta,\mathbb{T})$ is a log Fano cone singularity over an uncountable algebraically closed field, and $\nu_0$ is a $\mathbb{T}$-invariant valuation with center $x$ and $A_{X,\Delta}(\nu_0)<\infty$, then the value
$$ \delta(X,\Delta;\nu_0):=\inf_{\nu\in \mathrm{Val}^{\mathbb{T},*}_{X,\ni x}}\frac{A_{X,\Delta}(\nu)}{S(\nu_0;\nu)}$$
admits a minimum. The proof uses the generic limit argument. Note that there is a counterexample if we discard the log Fano cone structure.
\end{abstract}

\maketitle
\allowdisplaybreaks

\section{Introduction}
In this paper, every field is characteristic $0$.

\smallskip

The problem of the existence of a minimizer of an invariant is important in valuation theory because valuation theory provides a natural framework for attaining the minimum. In \cite{JM12}, Jonsson and Mustață stated and proved the existence of a minimizer of a log canonical threshold of a graded sequence of ideals; in \cite{Blu18}, Blum proved the existence of a minimizer of normalized local volume; and in \cite{BJ20}, they proved the existence for $\alpha$- and $\delta$-invariants. In \cite{CJKL25}, it is proved that there exists a minimizer of $\mathrm{lct}_{\sigma}(X,\Delta,D)$ for a klt pair $(X,\Delta)$ and a pseudo-effective $\Q$-Cartier $\Q$-divisor $D$ on $X$. In the theory of $K$-stability, finding a minimizer can be regarded as part of the stable degeneration theorem.

\smallskip

In this paper, we are dealing with so-called ``local $\delta$-invariant'' $\delta(X,\Delta;\nu_0)$, which is first defined in \cite{XZ20}. For $\delta(X,\Delta;\nu_0)\ge 1$, they showed that $\nu_0$ is a minimizer of $\delta(X,\Delta;\nu_0)$; however, $\delta(X,\Delta;\nu_0)<1$, it is hard to expect that some $\nu\in \mathrm{Val}_{X,x}$ would minimize $\delta(X,\Delta;\nu_0)$. Hence, we include any valuation $\mathrm{Val}_{X,\ni x}$ whose center is specialized to $x$.

\smallskip

We are interested in an analog of the local $\delta$-invariant for a log Fano cone singularity $x\in (X,\Delta,\mathbb{T})$. A log Fano cone singularity is a local analog of a log Fano pair. For example, if $(X,\Delta)$ is log Fano, then the affine cone
$$ C:=\Spec \bigoplus_{m\ge 0}H^0(X,\floor{-m(K_X+\Delta)}) $$
admits a log Fano cone singularity. An advantage of dealing with log Fano cone singularities is that there is an abundant theory on local K-stability that uses log Fano cone singularities as a language (cf. \cite{Li17, LX18, LX20, Zhu24a, Zhu24, XZ25, LW24, Wan25}).

\smallskip

The aim of this paper is to extend the definition of $\delta(X,\Delta;\xi)$ in \cite{Hua22, LW24, Wan25} by replacing $\mathrm{wt}_{\xi}$ with $\mathbb{T}$-invariant valuations $\nu_0$ and exploring the minimum attainment property of $\delta(X,\Delta;\nu_0)$; we use $\mathrm{Val}^{\mathbb{T},*}_{X,\ni x}$ in Definition \ref{definition} instead of $\mathrm{Val}^{\mathbb{T},*}_{X,x}$ (cf. Remark \ref{pohang}), and the definitions are equivalent because of Corollary \ref{sudosan}. The main theorem of this paper is the following:

\begin{theorem} \label{누영이}
Let $x\in (X,\Delta;\mathbb{T})$ be a $n$-dimensional log Fano cone singularity over an uncountable algebraically closed field $k$. Let $\nu_0\in \mathrm{Val}^{\mathbb{T},*}_{X,x}$ be a $\mathbb{T}$-invariant valuation with $A_{X,\Delta}(\nu_0)<\infty$.

\smallskip

Then, there exists a $\mathbb{T}$-invariant quasi-monomial valuation $\nu'\in \mathrm{Val}^{\mathbb{T}}_{X,\ni x}$ that is minimizing $\delta(X,\Delta;\nu_0)$.
\end{theorem}

Note that the result is proven in \cite[Theorem 5.0.4]{Hua22} in the case of $\nu_0=\mathrm{wt}_{\xi}$ for some Reeb vector $\xi$.

\smallskip

Our approach uses the generic limit argument (\cite{dFM09, dFEM10, Blu18, BJ20}), and in particular, we follow the approach of \cite{BJ20}. A challenge in adapting the generic limit argument from \cite{BJ20} is that, in our setting, there is no counterpart to \cite[Proposition 6.3]{BJ20}. To overcome this, we first descend all data associated with $(X,\Delta,\mathbb{T})$ and $\nu_0$ to a countable algebraically closed subfield $k_0\subseteq k$, and then apply Lemma \ref{duidui} to obtain an analog of \cite[Proposition 6.7]{BJ20} (or \cite[Lemma 3.21]{BLXZ25}). The argument is already presented in \cite[Theorem 4.1]{Tot12} and \cite[Proposition 3.9]{LX20}.

\smallskip

The condition of being a log Fano cone singularity is crucial: there is a counterexample when we do not consider the log Fano cone structure, even when $\dim X=2$ (cf. Question \ref{chatchatchat}). Let us explain why our generic limit argument cannot work for the general, non log Fano cone case. Intuitively, the generic limit argument gives a formal filtration on $\widehat{R}$, and without the condition that $x\in (X,\Delta,\mathbb{T})$ is a log Fano cone singularity, we cannot descend the formal filtration to $R$. Moreover, this counterexample suggests that it may be quite difficult to adapt the approach based on Birkar’s boundedness of complements (see \cite{Bir19}) to establish Theorem \ref{누영이}, in the same way as in \cite{Xu20, BLX22, BLXZ23}.

\smallskip

Section \ref{2} gives the basic notations and definitions we need, and Section \ref{3} proves the main theorem. Section \ref{4} explains why Theorem \ref{누영이} is false when we discard the log Fano cone structure and proposes two questions.

\section*{Acknowledgements}
The author is partially supported by Samsung Science and Technology Foundation under Project Number SSTF-BA2302-03. The author thanks Sung Rak Choi for identifying typographical errors in the draft of this paper.

\section*{Statement of AI use}
In proving Theorem \ref{누영이}, the author used ChatGPT 5.6 Sol to search the literature and to check and refine the main argument developed by the author. The main argument of the proof of Theorem \ref{누영이} is due to the author. In contrast, the counterexample in Section \ref{4} was found and initially constructed with substantial assistance from ChatGPT 5.6 Sol Ultra during discussions with the author. The author has independently verified the counterexample and all mathematical claims in the manuscript and takes full responsibility for their correctness.

\section{Preliminaries} \label{2}

In this paper, \emph{variety} means reduced, irreducible, separated, and finite type scheme over an uncountable algebraically closed field $k$. Any $\Q$-divisor is $\Q$-Weil unless otherwise stated. Let us collect the basic notions and definitions which will be used in the paper. For more details, see \cite{KM98} or \cite{Fuj17}.

\begin{itemize}
\item Given a variety $X$ and an effective $\Q$-divisor $\Delta$ on it, we say that $(X,\Delta)$ is a \emph{couple}. A couple $(X,\Delta)$ is a \emph{pair} when $K_X+\Delta$ is $\Q$-Cartier.
\item Let us say that a projective morphism $f:X\to Y$ of normal projective varieties is a \emph{fibration} if $f_*\mathcal{O}_X=\mathcal{O}_Y$.
\item Consider $X$ as a variety. A morphism $f: X'\to X$ is a \emph{resolution} if it is a proper birational morphism and $X'$ is smooth. For a pair $(X, \Delta)$, the morphism $f:X'\to X$ is a \emph{log resolution} of $(X,\Delta)$ if $f$ is a proper birational morphism, and the union $\mathrm{Exc}(f) \cup \mathrm{Supp}\, \Delta$ forms a simple normal crossing (snc) divisor on $X'$.
\item Let $(X,\Delta)$ be a pair, and let $E$ be a prime divisor over $X$. Suppose $f:X'\to X$ is a proper birational morphism from a normal variety $X'$ that has $E$ as an effective $f$-exceptional divisor. We denote by \emph{log discrepancy} $A_{X,\Delta}(E)$ of $(X,\Delta)$ along $E$, and the definition is
    $$ A_{X,\Delta}(E):=\mathrm{mult}_E(K_{X'}-f^*(K_X+\Delta))+1.$$
    Note that the definition does not depend on the choice of $f$. If $A_{X,\Delta}(E)>0\,(\ge 0)$ for every prime divisor $E$ over $X$, then $(X,\Delta)$ is \emph{Kawamata log terminal (klt)} (resp. \emph{log canonical (lc)}).
    \item Let $X\to S$ be a morphism of schemes, and let $s\in S$ be a (not necessarily closed) point. Denote by $X_s$ the fiber product $X\times_S s$, and by $X_{\overline{s}}$ the fiber product $X\times_S \overline{s}$. We say that $X_{\overline{s}}$ is a \emph{geometric fiber} of $X\to S$.
\end{itemize}

\subsection{Valuations}
Let $X$ be a scheme irreducible, separated, and essentially of finite type over an algebraically closed field $k$ of characteristic $0$. A \emph{(real) valuation} of its function field $K(X)$ is a non-constant map
$$ \nu:K(X)^{\times}\to \R $$
satisfying
\begin{itemize}
    \item $\nu(fg)=\nu(f)+\nu(g)$;
    \item $\nu(f+g)\ge \min\{\nu(f),\nu(g)\}$;
    \item $\nu(a)=0$ for every $a\in k$.
\end{itemize}
We set $\nu(0)=+\infty$. A valuation $\nu$ gives rise to a valuation ring
$$ \mathcal{O}_{\nu}:=\{f\in K(X)\mid \nu(f)\ge 0\}.$$
We say a valuation $\nu$ is centered at a scheme-theoretic point $x=c_X(\nu)\in X$ if we have a local inclusion $\mathcal{O}_{X,x}\hookrightarrow \mathcal{O}_{\nu}$ of local rings. We say that a valuation $\nu$ of $K(X)$ is \emph{valuation over $X$} if the center of $\nu$ is in $X$. For a closed point $x\in X$, we denote by $\mathrm{Val}_{X,x}$ (resp. $\mathrm{Val}_{X,\ni x}$) the set of valuations $\nu$ over $X$ centered at $x\in X$ (resp. the set of valuations over $X$ with $x\in \overline{\{c_X(\nu)\}}$).

\smallskip

Let $X:=\Spec R$ be affine. For each valuation $\nu\in \mathrm{Val}_{X}$ and any non-negative real number $\lambda$, we define the \emph{valuative ideal}
$$ \mathfrak{a}_{\lambda}(\nu):=\{f\in R\mid \nu(f)\ge \lambda\}.$$
Note that this definition is nonstandard: in many papers, the valuative ideal is defined as
$$ \mathfrak{a}_{\lambda}(\nu):=\{f\in \mathcal{O}_{X,c_X(\nu)}\mid \nu(f)\ge \lambda\}. $$
 This definition makes sense for any scheme $X$ (not necessarily affine). We adopt a different, nonstandard definition here in order to simplify our proof of Theorem \ref{누영이}. We think \cite{LX20} also uses the nonstandard definition of valuative ideals.

\smallskip

Given a nonzero ideal $\mathfrak{a}\subseteq \mathcal{O}_X$, we may evaluate $\mathfrak{a}$ along $\nu$ by setting
$$ \nu(\mathfrak{a}):=\min\{\nu(f)\mid f\in \mathfrak{a}\cdot \mathcal{O}_{X,c_X(\nu)}\}. $$
We endow $\mathrm{Val}_X$ with the weakest topology such that, for every nonzero ideal $\mathfrak{a}\subseteq \mathcal{O}_X$, the map $\mathrm{Val}_X\to \R$ defined by $\nu\mapsto \nu(\mathfrak{a})$ is continuous. The subset $\mathrm{Val}_{X,x},\mathrm{Val}_{X,\ni x}\subseteq \mathrm{Val}_X$ is endowed with the subspace topology.

\smallskip

A model $(X',E)$ over a pair $(X,\Delta)$ consists of a projective birational morphism $\pi:X'\to X$ and a reduced divisor $E$ on $X'$. We call it a \emph{log smooth model} if $(X',\Supp(E+\mathrm{Exc}(\pi)+\pi^{-1}_*\Delta))$ is simple normal crossing.

\smallskip

A valuation $\nu:=\nu_{\alpha,\eta}$ over $X$ is called \emph{quasi-monomial} if there exists a log smooth model $(X',E:=E_1+\cdots+E_r)\to X$, a nonnegative number $r$ and a tuple $\alpha=(\alpha_1,\cdots,\alpha_r)$ such that
\begin{itemize}
    \item[(1)] $\bigcap^r_{i=1}E_i\ne \varnothing$;
    \item[(2)] There exists a component $C\subseteq \bigcap^r_{i=1}E_i$, such that around the generic point $\eta$ of $C$, $E_i$ is given by the local equation $(z_i=0)$ and
    \item[(3)] for any $f\in \mathcal{O}_{X',\eta}$, we may write $f=\sum_{\beta\in \Z^r_{\ge 0}}c_{\beta}z^{\beta}$ (where $z^{\beta}:=z^{\beta_1}_1\cdots z^{\beta_r}_r$ in $\widehat{\mathcal{O}_{X',\eta}}$ so that each $c_{\beta}$ is either zero or a unit), then
    $$ \nu(f)=\min\left\{\sum \alpha_i\beta_i\mid c_{\beta}\ne 0\right\}.$$
\end{itemize}
We define
$$ \mathrm{QM}_{\eta}(X',E):=\left\{\nu_{\alpha,\eta}\mid \alpha\in \R^r_{\ge 0}\right\}  $$
and
$$ \mathrm{QM}(X',E):=\bigcup_{\eta \in \bigcap^r_{i=1} E_i}\mathrm{QM}_{\eta}(X',E), $$
where $\eta$ runs through the generic points of $\bigcap^r_{i=1}E_i$.

\smallskip

For a quasi-monomial valuation $\nu_{(\alpha_1,\cdots,\alpha_r),\eta}$ with a log smooth model $(X',E=E_1+\cdots+E_r)\to (X,\Delta)$, we define
$$ A_{X,\Delta}(\nu_{(\alpha_1,\cdots,\alpha_r),\eta}):=\sum \alpha_j A_{X,\Delta}(E_j).$$
For every log smooth model $(X',E)$ over $X$, we can define the \emph{retraction map} $r_{(X',E)}:\mathrm{Val}_X\to \mathrm{QM}(X',E)$ as in \cite[4.3]{JM12}, and define
$$ A_{X,\Delta}(\nu):=\sup_{(X',E)}A_{X,\Delta}(r_{(X',E)}(\nu))\text{ for }\nu\in \mathrm{Val}_X,$$
where the supremum is taken over all log smooth pairs $(X',E)$ over $(X,\Delta)$.

\subsection{Multiplier ideal sheaf}
In this subsection, we recall the definition of \emph{multiplier ideal sheaf}.

\smallskip

Let $(X,\Delta)$ be a pair, $\mathfrak{a}$ be an ideal in $\mathcal{O}_X$, and let $f:X'\to X$ be a log resolution of $(X,\Delta,\mathfrak{a})$. Let $F$ be an effective Cartier divisor on $X'$ such that $\mathfrak{a}\mathcal{O}_{X'}=\mathcal{O}_{X'}(-F)$. Let $\lambda\ge 0$. We define the \emph{multiplier ideal sheaf} $\mathcal{J}(X,\Delta,\mathfrak{a}^{\lambda})$ as
$$ \mathcal{J}(X,\Delta,\mathfrak{a}^{\lambda})=f_*\mathcal{O}_{X'}(\ceil{K_{X'}-f^*(K_X+\Delta)-\lambda E}).$$
We define the \emph{log canonical threshold} as
$$ \mathrm{lct}(X,\Delta,\mathfrak{a}):=\inf\{\lambda\mid \mathcal{J}(X,\Delta,\mathfrak{a}^{\lambda})\ne \mathcal{O}_X\}.$$
It is well-known
$$ \mathrm{lct}(X,\Delta,\mathfrak{a})=\inf_{\nu\in \mathrm{Val}^*_X}\frac{A_{X,\Delta}(\nu)}{\nu(\mathfrak{a})}.$$
(cf. \cite[Corollary 6.9]{JM12}, \cite[Lemma 3.4]{Kim26}, and \cite[Theorem 3.7]{CJKL25}) Moreover, for a graded sequence of ideals $\mathfrak{a}_{\bullet}$ in $\mathcal{O}_X$, we can define
$$ \mathcal{J}(X,\Delta,\mathfrak{a}^{\lambda}_{\bullet}):=\bigcup_{m}\mathcal{J}(X,\Delta,\mathfrak{a}^{\frac{\lambda}{m}}_m)$$
for a sufficiently divisible $m$, and
$$ \mathrm{lct}(X,\Delta,\mathfrak{a}_{\bullet}):=\sup_{m\ge 1}m\cdot \mathrm{lct}(X,\Delta,\mathfrak{a}_m) $$
for a graded sequence of ideals $\mathfrak{a}_{\bullet}$ in $\mathcal{O}_X$. Note that
$$ \mathrm{lct}(X,\Delta,\mathfrak{a}_{\bullet})=\inf_{\nu\in \mathrm{Val}^*_X}\frac{A_{X,\Delta}(\nu)}{\nu(\mathfrak{a}_{\bullet})},$$
where $\nu(\mathfrak{a}_{\bullet})=\inf_m\frac{\nu(\mathfrak{a}_m)}{m}$ (cf. \cite[Chapter 1]{Xu25}).



\subsection{{Local $S$-invariant}}
In this subsection, we extend the definition of \emph{local $S$-invariant} in \cite[Lemma-Definition 3.1]{XZ20} to every $\nu\in \mathrm{Val}_{X,\ni x}$.

\begin{definition} \label{yongheungdong}
Let $x\in (X:=\Spec R,\Delta)$ be an $n$-dimensional klt singularity, fix $\nu_0\in \mathrm{Val}^*_{X,x}$ with $A_{X,\Delta}(\nu_0)<\infty$, and let $\nu\in \mathrm{Val}^*_{X,\ni x}$.

\smallskip

Define
$$ \mathrm{vol}(\nu_0;\nu):=\limsup_{\substack{r\to \infty \\ r\in \R_{>0}}}\frac{\ell_R\left(\frac{\mathfrak{a}_{r}(\nu)+\mathfrak{a}_r(\nu_0)}{\mathfrak{a}_r(\nu_0)}\right)}{\frac{r^n}{n!}},$$
and
$$ S(\nu_0;\nu):=\frac{n+1}{n}\frac{A_{X,\Delta}(\nu_0)}{\mathrm{vol}(\nu_0)}\int^{\infty}_0\mathrm{vol}\left(\nu_0;\frac{\nu}{t}\right)\,\mathrm{d}t.$$
\end{definition}

Unlike the case of $\nu\in \mathrm{Val}_{X,x}$ (cf. \cite[Lemma-Definition 3.1]{XZ20}), we do not know whether the limit exists in general. Hence, we use the limit supremum of real numbers.

\begin{lemma} \label{3.2}
For all $p>0$ and all $t\ge 0$,
$$ \mathrm{vol}\left(\nu_0;\frac{\nu}{t}\right)=\limsup_{\substack{q\to \infty \\ q\in \Z_{>0}}}\frac{\ell_R\left(\frac{\mathfrak{a}_{pqt}(\nu)+\mathfrak{a}_{pq}(\nu_0)}{\mathfrak{a}_{pq}(\nu_0)}\right)}{\frac{(pq)^n}{n!}}.$$
\end{lemma}

\begin{proof}
Note that the $\ge $ side is trivial, and hence we will only prove the $\le $ side.

\smallskip

For $0<s\le r$, $\mathfrak{a}_{rt}(\nu)\subseteq \mathfrak{a}_{st}(\nu)$, and $\mathfrak{a}_r(\nu_0)\subseteq \mathfrak{a}_s(\nu_0)$. Hence,
$$ \mathfrak{a}_{rt}(\nu)+\mathfrak{a}_r(\nu_0)\subseteq \mathfrak{a}_{st}(\nu)+\mathfrak{a}_s(\nu_0).$$
Hence,
$$ 
\begin{aligned}
\ell_R\left(\frac{\mathfrak{a}_{rt}(\nu)+\mathfrak{a}_r(\nu_0)}{\mathfrak{a}_r(\nu_0)}\right)&\le \ell_R\left(\frac{\mathfrak{a}_{st}(\nu)+\mathfrak{a}_s(\nu_0)}{\mathfrak{a}_r(\nu_0)}\right)
\\ &=\ell_R\left(\frac{\mathfrak{a}_{st}(\nu)+\mathfrak{a}_s(\nu_0)}{\mathfrak{a}_s(\nu_0)}\right)+\ell_R(\mathfrak{a}_s(\nu_0)/\mathfrak{a}_r(\nu_0)).
\end{aligned}
$$
Now, let $r_j\to \infty$ be a sequence such that
$$\mathrm{vol}\left(\nu_0;\frac{\nu}{t}\right)=\lim_{i\to \infty}\frac{\ell_R\left(\frac{\mathfrak{a}_{r_it}(\nu)+\mathfrak{a}_{r_i}(\nu_0)}{\mathfrak{a}_{r_i}(\nu_0)}\right)}{\frac{r^n_i}{n!}},$$
$q_j:=\floor{\frac{r_j}{p}}$, and $s_j=pq_j$. Then, $s_j\le r_j<s_j+p$ and $\frac{s_j}{r_j}\to 1$. Thus,
$$ \frac{\ell_R\left(\frac{\mathfrak{a}_{r_it}(\nu)+\mathfrak{a}_{r_i}(\nu_0)}{\mathfrak{a}_{r_i}(\nu_0)}\right)}{\frac{r^n_i}{n!}}\le \left(\frac{s_j}{r_j}\right)^n\frac{\ell_R\left(\frac{\mathfrak{a}_{s_it}(\nu)+\mathfrak{a}_{s_i}(\nu_0)}{\mathfrak{a}_{s_i}(\nu_0)}\right)}{\frac{s^n_i}{n!}}+\frac{\ell_R(\mathfrak{a}_{s_j}(\nu_0)/\mathfrak{a}_{r_j}(\nu_0))}{\frac{r^n_j}{n!}}.$$
The latter goes to $0$ as $j\to \infty$:
$$ 
\begin{aligned}
\limsup_{j\to \infty}\frac{\ell_R(\mathfrak{a}_{s_j}(\nu_0)/\mathfrak{a}_{r_j}(\nu_0))}{\frac{r^n_j}{n!}}&=\limsup_{j\to \infty}\frac{\ell_R(R/\mathfrak{a}_{r_j}(\nu_0))-\ell_R(R/\mathfrak{a}_{s_j}(\nu_0))}{\frac{r^n_j}{n!}}
\\ &\le \limsup_{j\to \infty}\frac{\ell_R(R/\mathfrak{a}_{r_j}(\nu_0))}{\frac{r^n_j}{n!}}-\liminf_{j\to \infty}\left(\frac{s_j}{r_j}\right)^n\frac{\ell_R(R/\mathfrak{a}_{s_j}(\nu_0))}{\frac{s^n_j}{n!}}
\\ &=0,
\end{aligned}
$$
and similarly, the $\liminf$ is $\ge 0$. Hence, by letting $i\to \infty$,
$$ \mathrm{vol}\left(\nu_0;\frac{\nu}{t}\right)\le \limsup_{i\to \infty}\frac{\ell_R\left(\frac{\mathfrak{a}_{s_it}(\nu)+\mathfrak{a}_{s_i}(\nu_0)}{\mathfrak{a}_{s_i}(\nu_0)}\right)}{\frac{s^n_i}{n!}}\le \limsup_{\substack{q\to \infty \\ q\in \R_{>0}}}\frac{\ell_R\left(\frac{\mathfrak{a}_{pqt}(\nu)+\mathfrak{a}_{pq}(\nu_0)}{\mathfrak{a}_{pq}(\nu_0)}\right)}{\frac{(pq)^n}{n!}}.$$
We complete the proof.
\end{proof}

Let us prove the following lemma.

\begin{lemma} \label{HLQ}
Let $x\in (X:=\Spec R,\Delta)$ be a klt germ, and let $f:(Y:=\Spec S,\Delta_Y)\to (X,\Delta)$ be an étale morphism. Let $y\in Y$ be a closed point with $f(y)=x$, and $\nu,\nu_0\in \mathrm{Val}_{Y,y}$ valuations.

\smallskip

Then, we obtain that
\begin{itemize}
    \item[(1)] the restriction map $\mathrm{Val}_{Y,y}\to \mathrm{Val}_{X,x}$ gives an isomorphism,
    \item[(2)] we have $A_{Y,\Delta_Y}(\nu)=A_{X,\Delta}(\nu|_{K(X)})$
    \item[(3)] under the isomorphism, $\mathrm{vol}(\nu;\nu_0)=\mathrm{vol}(\nu|_{K(X)};\nu_0|_{K(X)})$.
\end{itemize}
\end{lemma}

\begin{proof}
(1) and (2) are already proven in \cite[Proposition 2.24]{HLQ23}. Note that the proof of the pair version of \cite[Proposition 2.24 (1)]{HLQ23} is almost verbatim to \cite[The proof of Proposition 2.24 (1)]{HLQ23}. For (2), let $\widehat{R}$ be the completion of $R$ along $\mathfrak{m}_x$. Then,
$$
\begin{aligned}
\ell_S\left(\frac{\mathfrak{a}_r(\nu)+\mathfrak{a}_r(\nu_0)}{\mathfrak{a}_r(\nu_0)}\right)&=\ell_{\widehat{R}}\left(\frac{\mathfrak{a}_r(\nu)\widehat{R}+\mathfrak{a}_r(\nu_0)\widehat{R}}{\mathfrak{a}_r(\nu_0)\widehat{R}}\right)
\\ &=\ell_{\widehat{R}}\left(\frac{\mathfrak{a}_r(\nu|_{K(X)})\widehat{R}+\mathfrak{a}_r(\nu_0|_{K(X)})\widehat{R}}{\mathfrak{a}_r(\nu_0|_{K(X)})\widehat{R}}\right)
\\ &=\ell_R\left(\frac{\mathfrak{a}_r(\nu|_{K(X)})+\mathfrak{a}_r(\nu_0|_{K(X)})}{\mathfrak{a}_r(\nu_0|_{K(X)})}\right).
\end{aligned}
$$
We complete the proof.
\end{proof}

\subsection{Log Fano cone singularities}
In this subsection, we define the notion of \emph{log Fano cone singularity}. For more details, see \cite{Zhu24, Liu26}.

\begin{definition}
Let $x\in (X:=\Spec R,\Delta)$ be an affine normal pair of dimension $n$, and let $\mathbb{T}:=\mathbb{G}^r_m$. A $\mathbb{T}$-action on $(X,\Delta)$ is called \emph{good} if it is effective, $x$ is a $\mathbb{T}$-fixed point, and the closure of any $\mathbb{T}$-orbit.
\end{definition}

Denote by $X:=\Spec R$. Then, we have
$$ R:=\bigoplus_{\alpha\in M}R_{\alpha},$$
where $M:=\mathrm{Hom}(\mathbb{T},\mathbb{G}_m)\cong \Z^r$ is the weight lattice. We denote by $N:=M^{\vee}$ the coweight lattice. Note that $R_0=k$ and every $R_{\alpha}$ is finite-dimensional since the $\mathbb{T}$-action is good. We define the \emph{Reeb cone} $\mathbf{t}^+_{\R}\subseteq N_{\R}$ by
$$ \mathbf{t}^+_{\R}:=\{\xi\in N_{\R}\mid \langle \xi,\alpha\rangle >0\text{ for every }\alpha\ne 0\text{ with }R_{\alpha}\ne 0\}.$$
For any $\xi\in \mathbf{t}^+_{\R}$, we can define a valuation $\mathrm{wt}_{\xi}$ by setting
$$ \mathrm{wt}_{\xi}(f):=\min\{\langle \xi,\alpha\rangle \mid \alpha\in M,\,f_{\alpha}\ne 0\},$$
where $f\in R$.

\begin{definition}
Let $x\in (X,\Delta)$ be a klt singularity with a nontrivial good $\mathbb{T}$-action. If $\Delta$ is $\mathbb{T}$-invariant, then we say the triple \emph{log Fano cone singularity} $x\in (X,\Delta;\mathbb{T})$ is a \emph{log Fano cone singularity}.
\end{definition}

Let $x\in (X:=\Spec R,\Delta)$ admit a good torus $\mathbb{T}:=\mathbb{G}^r_m$ action. Let $R=\bigoplus_{\alpha}R_{\alpha}$ be the weight decomposition. We say a valuation $\nu$ over $X$ is $\mathbb{T}$\emph{-invariant} if
$$ \nu\left(\sum_{\alpha}f_{\alpha}\right)=\min_{\alpha}\nu(f_{\alpha})$$
for $f_{\alpha}\in R_{\alpha}$.
We denote by $\mathrm{Val}^{\mathbb{T}}_{X,\ni x}$ the subset of $\mathbb{T}$-invariant valuations.

\smallskip

A $\mathbb{T}$\emph{-invariant} filtration on $R$ is a family of vector subspaces $\mathcal{F}^{p}R=\bigoplus_{\alpha} \mathcal{F}^pR_{\alpha}$ of $R$ for $p\in \R_{\ge 0}$ such that
\begin{itemize}
    \item[(1)] $\mathcal{F}^0R=R$;
    \item[(2)] $\mathcal{F}^qR\subseteq \mathcal{F}^pR$ if $p\le q$;
    \item[(3)] $\mathcal{F}^pR=\bigcap_{q<p}\mathcal{F}^qR$;
    \item[(4)] $\mathcal{F}^pR_{\alpha}\cdot \mathcal{F}^qR_{\beta}\subseteq \mathcal{F}^{p+q}R_{\alpha+\beta}$.
\end{itemize}
Note that every $\mathbb{T}$-invariant valuation gives a $\mathbb{T}$-invariant filtration by valuative ideals. Moreover, a filtration $\mathcal{F}^{\bullet}$ gives us a graded sequence of ideals $\{\mathcal{F}^p\}_{p\in \Z_{>0}}$, and we denote the graded sequence of ideals as $\mathcal{F}^{\bullet}$ by an abuse of notation. Note that we do not impose any condition for $\mathcal{F}^p$ to be $\mathfrak{m}_x$-primary.

\begin{definition} \label{definition}
Let $x\in (X,\Delta;\mathbb{T})$ be a log Fano cone singularity, and let $\nu_0$ be a $\mathbb{T}$-invariant valuation with center $x$ and $A_{X,\Delta}(\nu_0)<\infty$. We define
$$ \delta(X,\Delta;\nu_0):=\inf_{\nu\in \mathrm{Val}^{\mathbb{T},*}_{X,\ni x}}\frac{A_{X,\Delta}(\nu)}{S(\nu_0;\nu)}.$$
\end{definition}

\subsection{{$\mathbb{T}$-equivariant quasi-monomiality theorem}}
We believe this lemma is well-known for experts (see, for instance, \cite[Remark 3.9]{BLXZ23} or \cite[Proof of Theorem 1.1]{Wan25}), but we were unable to find a place in the literature where it is explicitly proved. For this reason, we provide a proof here.

\begin{lemma} \label{hevensteelyard}
Let $x\in (X:=\Spec R,\Delta,\mathbb{T})$ be a log Fano cone singularity, and let $\mathfrak{a}_{\bullet}$ be a $\mathbb{T}$-invariant graded sequence of ideals in $\mathcal{O}_X$.

\smallskip

Then, there exists a $\mathbb{T}$-invariant quasi-monomial valuation computing $\mathrm{lct}(X,\Delta,\mathfrak{a}_{\bullet})$.
\end{lemma}

\begin{proof}
Let $N:=N(\mathbb{T})=\mathrm{Hom}(\mathbb{G}_m,\mathbb{T})$ be the co-weight lattice and $M:=N^*$ the weight lattice. We have a weight decomposition
$$ R=\bigoplus_{\alpha\in M}R_{\alpha}, \,\,\,R_0=k,\,\,\, \mathfrak{m}_x=\bigoplus_{\alpha \ne 0}R_{\alpha}.$$
Let $\nu$ be a quasi-monomial valuation computing $\mathrm{lct}(X,\Delta,\mathfrak{a}_{\bullet})$ (cf. \cite[Theorem 1.1]{Xu20}). For every $f\in R$, we can write $f$ as $f=\sum_{\alpha} f_{\alpha}$ for $f_{\alpha}\in R_{\alpha}$. Define
$$ \nu^{\mathbb{T}}(f):=\min_{f_{\alpha}\ne 0}\nu(f_{\alpha}).$$
\begin{itemize}
    \item[(1)] $\nu^{\mathbb{T}}$ is a well-defined valuation and $\mathbb{T}$-invariant.
\end{itemize}
Let us define $z^{\alpha}$ to be a character monomial in $k[M]$, and define
$$ \tilde{\nu}\left(\sum_{\alpha}f_{\alpha}z^{\alpha}\right)=\min_{\alpha}\nu(f_{\alpha}).$$
Then, it is a valuation on the fraction field of $K(X)[M]$. Moreover, define
$$ \rho:R\to R\otimes_k k[M], \,\,\, \rho(f)=\sum_{\alpha}f_{\alpha}z^{\alpha}. $$
Then, we can extend $\rho$ to $K(X)$, and moreover, $\nu^{\mathbb{T}}=\tilde{\nu}\circ \rho$. Therefore, $\nu^{\mathbb{T}}$ is a well-defined valuation.

\smallskip

Moreover, for any $t\in \mathbb{T}(k)$,
$$ \nu^{\mathbb{T}}(t\cdot f)=\min_{\alpha}\nu(\alpha(t)f_{\alpha})=\min_{\alpha}(f_{\alpha})=\nu^{\mathbb{T}}(f).$$
Thus, $\nu^{\mathbb{T}}$ is $\mathbb{T}$-invariant.

\begin{itemize}
    \item[(2)] For every $\mathbb{T}$-invariant ideal $\mathfrak{a}$, $\nu^{\mathbb{T}}(\mathfrak{a})=\nu(\mathfrak{a})$.
\end{itemize}

For every $f\in \mathfrak{a}$, we can write $f=\sum_{\alpha}f_{\alpha}$ such that $f_{\alpha}\in \mathfrak{a}_{\alpha}:=\mathfrak{a}\cap R_{\alpha}$. There exist $f\in \mathfrak{a}$ and $\beta$ such that 
$$\nu^{\mathbb{T}}(\mathfrak{a})=\nu^{\mathbb{T}}(f)=\nu(f_{\beta})\ge \nu(\mathfrak{a}).$$
Moreover, there exist $g\in \mathfrak{a}$, $g=\sum_{\alpha}g_{\alpha}$, and $\gamma$ such that
$$ \nu(\mathfrak{a})=\nu(g)=\nu(g_{\beta})\ge \min_{\alpha}\nu(g_{\alpha})=\nu^{\mathbb{T}}(g)\ge \nu^{\mathbb{T}}(\mathfrak{a}).$$

\begin{itemize}
    \item[(3)] $A_{X,\Delta}(\nu^{\mathbb{T}})\le A_{X,\Delta}(\nu)$.
\end{itemize}

Let $a:T\times X\to X$ be $(t,x)\mapsto tx$, and $p_2:T\times X\to X$ be the projection. Let $(X',E)\to (X,\Delta)$ be a log-smooth model, $\eta\in \bigcap_{j\in J} E_j$ the generic point of an irreducible component, and let $\alpha:=(\alpha_1,\cdots,\alpha_r)$ be a tuple such that $\nu=\nu_{\alpha,\eta}$. Let $\eta_{\mathbb{T}}$ be the generic point of an irreducible component of $\bigcap_{j\in J}T\times E_j$ corresponding to $\eta$. Then, $\tilde{\nu}=\nu_{\alpha,\eta_{\mathbb{T}}}$, and thus
$$ A_{X,\Delta}(\nu)=\sum \alpha_iA_{X,\Delta}(E_i)=\sum \alpha_i A_{T\times X,p^*_2\Delta}(E_i)=A_{T\times X,p^*_2\Delta}(\tilde{\nu}).$$
Since $\Delta$ is $\mathbb{T}$-invariant, $a^*\Delta=p^*_2\Delta$. Moreover, since $a$ is smooth, the singular version of \cite[Proposition 5.13]{JM12} (see also \cite[Proof of Theorem 23]{Liu18}) gives
$$ A_{T\times X,p^*_2\Delta}(\tilde{\nu})=A_{T\times X,a^*\Delta}(\tilde{\nu})\ge A_{X,\Delta}(\nu^{\mathbb{T}}).$$
Note that the restriction of $\tilde{\nu}$ along $(p_2)_*K(X)\subseteq K(T\times X)$ is $\nu$, and the restriction of $\tilde{\nu}$ along $a_*K(X)\subseteq K(T\times X)$ is $\nu^{\mathbb{T}}$.

\begin{itemize}
\item[(4)] $\nu^{\mathbb{T}}$ is quasi-monomial.
\end{itemize}

Since $\tilde{\nu}$ is quasi-monomial, we obtain that
$$ \mathrm{rat.rk}(\tilde{\nu})+\mathrm{trdeg}_k\kappa(\tilde{\nu})=\dim X+\dim \mathbb{T}.$$
Moreover,
$$
\begin{aligned}\mathrm{rat.rk}(\tilde{\nu})+\mathrm{trdeg}_k\kappa(\tilde{\nu})&\overset{(\whitestar)}{\le} \mathrm{rat.rk}(\nu^{\mathbb{T}})+\mathrm{trdeg}_k\kappa(\nu^{\mathbb{T}})+\mathrm{trdeg}_{a^*K(X)}K(T\times X)
\\ & =\mathrm{rat.rk}(\nu^{\mathbb{T}})+\mathrm{trdeg}_k\kappa(\nu^{\mathbb{T}})+\dim \mathbb{T},
\end{aligned}
$$
where $(\whitestar)$ comes from \cite[Theorem 1.24]{Xu25}. Hence,
$$ \mathrm{rat.rk}(\nu^{\mathbb{T}})+\mathrm{trdeg}_k\kappa(\nu^{\mathbb{T}})\ge \dim X.$$
Thus, by \cite[Proposition 1.28]{Xu25}, we complete the proof of (4).

\smallskip

Since $\nu$ computes $\mathrm{lct}(X,\Delta,\mathfrak{a}_{\bullet})$, we obtain
$$ \mathrm{lct}(X,\Delta,\mathfrak{a}_{\bullet})=\frac{A_{X,\Delta}(\nu)}{\nu(\mathfrak{a}_{\bullet})}\overset{(\star)}{\ge} \frac{A_{X,\Delta}(\nu^{\mathbb{T}})}{\nu^{\mathbb{T}}(\mathfrak{a}_{\bullet})},$$
where we used (2) and (3) in $(\star)$. We complete the proof.
\end{proof}

We can define $\delta(X,\Delta;\nu_0)$ using only valuations centered on $x$.

\begin{corollary} \label{sudosan}
Let $x\in (X:=\Spec R,\Delta,\mathbb{T})$ be a log Fano cone singularity. Let $\nu_0\in \mathrm{Val}^{\mathbb{T},*}_{X,x}$ be a $\mathbb{T}$-invariant valuation with $A_{X,\Delta}(\nu_0)<\infty$.

\smallskip

Then, we obtain that
$$ \delta(X,\Delta;\nu_0)=\inf_{\nu\in \mathrm{Val}^{\mathbb{T},*}_{X,x}}\frac{A_{X,\Delta}(\nu)}{S(\nu_0;\nu)}. $$
\end{corollary}

\begin{proof}
Let $N:=N(\mathbb{T})=\mathrm{Hom}(\mathbb{G}_m,\mathbb{T})$ be the co-weight lattice and $M:=N^*$ the weight lattice. We have a weight decomposition
$$ R=\bigoplus_{\alpha\in M}R_{\alpha}, \,\,\,R_0=k,\,\,\, \mathfrak{m}_x=\bigoplus_{\alpha \ne 0}R_{\alpha}.$$

\smallskip

Let $\delta_c(X,\Delta;\nu_0)$ be the right-hand side. Since $\mathrm{Val}^{\mathbb{T},*}_{X,x}\subseteq \mathrm{Val}^{\mathbb{T},*}_{X,\ni x}$, we obtain $\delta(X,\Delta;\nu_0)\le \delta_c(X,\Delta;\nu_0)$.

\smallskip

Let $\xi\in \mathfrak{t}^+_{\R}$ be a Reeb vector, and let $\nu\in \mathrm{Val}^{\mathbb{T},*}_{X,\ni x}$ be a valuation that is not centered on $x$. Then, $\nu(\mathfrak{m}_x)=0$. Let $\varepsilon>0$, and let us define the valuation $\nu_{\varepsilon}$ by
$$ \nu_{\varepsilon}\left(\sum_{\alpha}f_{\alpha}\right)=\min_{f_{\alpha}\ne 0}(\nu(f_{\alpha})+\varepsilon \langle \xi,\alpha\rangle),\,\,\,f_{\alpha}\in R_{\alpha}.$$
If $f_{\alpha}\in \mathfrak{m}_x$, then $\alpha\ne 0$, and $\nu_{\varepsilon}(f_{\alpha})>0$. Hence, $\nu_{\varepsilon}\in \mathrm{Val}^{\mathbb{T},*}_{X,x}$. Moreover, $\nu_{\varepsilon}\to \nu$ as $\varepsilon \to 0$ pointwise.

\smallskip

There exists $C_{\xi}$ such that for every $f\in R$,
\begin{equation} \label{용흥동}
\mathrm{wt}_{\xi}(f)\le C_{\xi}\nu_0(f)\,\,\,(0\ne f\in R)
\end{equation}
(cf. \cite[Proposition 2.2]{XZ20}).

\smallskip

Let us prove that for a positive real $r$ and $s>C_{\xi}\varepsilon$,
$$ \mathfrak{a}_{rs}(\nu)+\mathfrak{a}_r(\nu_0)\subseteq \mathfrak{a}_{rs}(\nu_{\varepsilon})+\mathfrak{a}_r(\nu_0)\subseteq \mathfrak{a}_{r(s-C_{\xi}\varepsilon)}(\nu)+\mathfrak{a}_r(\nu_0).$$
The first inclusion comes from $\nu_{\varepsilon}\ge \nu$. For the second, let $f\in \mathfrak{a}_{rs}(\nu_{\varepsilon})$, and let $f=\sum_{\alpha}f_{\alpha}$ be the homogeneous decomposition. Then,
\begin{itemize}
    \item if $\nu_0(f_{\alpha})\ge r$, then $f_{\alpha}\in \mathfrak{a}_r(\nu_0)$,
    \item if $\nu_0(f_{\alpha})<r$, then
    $$ 
    \begin{aligned}
    \nu(f_{\alpha})&=\nu_{\varepsilon}(f_{\alpha})-\varepsilon \mathrm{wt}_{\xi}(f_{\alpha}) &
    \\ & \ge rs-\varepsilon C_{\xi}\nu_0(f_{\alpha}) & (\mbox{\ding{81}})
    \\ &>r(s-C_{\xi}\varepsilon), & 
    \end{aligned}$$
    and therefore $f_{\alpha}\in \mathfrak{a}_{r(s-C_{\xi}\varepsilon)}(\nu)$ (we used (\ref{용흥동}) in $(\mbox{\ding{81}})$).
\end{itemize}

Hence, we obtain
$$ \ell_R\left(\frac{\mathfrak{a}_{rs}(\nu)+\mathfrak{a}_r(\nu_0)}{\mathfrak{a}_r(\nu_0)}\right)\le \ell_R\left(\frac{\mathfrak{a}_{rs}(\nu_{\varepsilon})+\mathfrak{a}_r(\nu_0)}{\mathfrak{a}_r(\nu_0)}\right)\le \ell_R\left(\frac{\mathfrak{a}_{r(s-C_{\xi}\varepsilon)}(\nu)+\mathfrak{a}_r(\nu_0)}{\mathfrak{a}_r(\nu_0)}\right).$$
Therefore,
$$ \mathrm{vol}\left(\nu_0;\frac{\nu}{s}\right)\le \mathrm{vol}\left(\nu_0;\frac{\nu_{\varepsilon}}{s}\right)\le \mathrm{vol}\left(\nu_0;\frac{\nu}{s-C_{\xi}\varepsilon}\right).$$
Thus,
$$ 0\le S(\nu_0;\nu_{\varepsilon})-S(\nu_0;\nu)\le \frac{n+1}{n}A_{X,\Delta}(\nu_0)C_{\xi}\varepsilon.$$
Hence, $S(\nu_0;\nu_{\varepsilon})\to S(\nu_0;\nu)$ as $\varepsilon\to 0^+$.

\smallskip

Let $\nu_i\in \mathrm{Val}^{\mathbb{T},*}_{X,\ni x}$ be a $\mathbb{T}$-invariant valuation such that
$$ \frac{A_{X,\Delta}(\nu_i)}{S(\nu_0;\nu_i)}\to \delta(X,\Delta;\nu_0).$$ 
We may assume $\nu_i$ to be quasi-monomial; indeed, let $\nu'_i$ be a $\mathbb{T}$-invariant quasi-monomial valuation computing $\mathrm{lct}(X,\Delta,\mathfrak{a}_{\bullet}(\nu_i))$ (cf. Lemma \ref{hevensteelyard}). Under the scaling $\nu'_i(\mathfrak{a}_{\bullet}(\nu_i))=1$, we obtain $\nu'_i\ge \nu_i$, and so $S(\nu_0;\nu'_i)\ge S(\nu_0;\nu_i)$. Moreover,
$$ A_{X,\Delta}(\nu'_i)=\mathrm{lct}(X,\Delta,\mathfrak{a}_{\bullet}(\nu_i))\le A_{X,\Delta}(\nu_i).$$
Combining those two results, we obtain
$$ \frac{A_{X,\Delta}(\nu_i)}{S(\nu_0;\nu_i)}\ge \frac{A_{X,\Delta}(\nu'_i)}{S(\nu_0;\nu'_i)}.$$
Moreover, since any prime ideal of the center of $\nu'_i$ is $\mathbb{T}$-invariant, the closure of the center contains $x$ because the unique fixing point of $\mathbb{T}$-action is $x$. Hence there is no problem replacing $\nu_i$ with $\nu'_i$.

\smallskip

If $c_X(\nu_i)=x$ for infinitely many $i$, then there is nothing to do. If $c_X(\nu_i)\ne x$ for every $i$, then by \cite[Lemma 2.8]{LW24},
$$ A_{X,\Delta}(\nu_{i\varepsilon})=A_{X,\Delta}(\nu_i)+\varepsilon A_{X,\Delta}(\mathrm{wt}_{\xi}).$$
Hence,
$$ \delta_c(X,\Delta;\nu_0)\le \frac{A_{X,\Delta}(\nu_{i\varepsilon})}{S(\nu_0;\nu_{i\varepsilon})}\to \frac{A_{X,\Delta}(\nu_i)}{S(\nu_0;\nu_i)}\text{ as }\varepsilon\to 0^+.$$
Therefore, we may choose $\varepsilon_i>0$ such that
$$\delta_c(X,\Delta;\nu_0)\le \frac{A_{X,\Delta}(\nu_{i\varepsilon_i})}{S(\nu_0;\nu_{i\varepsilon_i})}\to \delta(X,\Delta;\nu_0)\text{ as }i\to \infty. $$
We have obtained the opposite inequality.
\end{proof}

\begin{remark} \label{pohang}
As noted in \cite[Remark 3.7]{Wan25}, there may be no computing valuation of $\delta(X,\Delta;\nu_0)$ centered on $x$; this is the reason why we use $\mathrm{Val}^{\mathbb{T}}_{X,\ni x}$, rather than $\mathrm{Val}^{\mathbb{T}}_{X,x}$. This is also the reason why our canonical definition of $\delta(X,\Delta;\nu_0)$ includes all valuations with $x\in \overline{\{c_X(\nu)\}}$.
\end{remark}

\section{Proof of the Main Theorem} \label{3}

The aim of this section is to prove the main theorem of the paper, Theorem \ref{누영이}.

\subsection{Uniform Fujita approximation}

The Fujita approximation theorem is first proven in \cite{Fuj94}, and this idea was extended to provide a uniform approximation on $S$-invariant in \cite[Section 5]{BJ20}. In this subsection, we extend the idea to $S(\nu_0;\nu)$.

\smallskip

Let us explain the difference in the construction of $\tilde{S}_m(\nu,L)$ in \cite[Section 5]{BJ20} compared to $\overline{S}_m(\nu_0;\nu)$. In \cite[Section 5]{BJ20}, the reason we take integral closure in the definition is to use \cite[Proposition 6.3]{BJ20} and ensure
$$ S(\mathcal{F})\ge \liminf_{m\to \infty}\liminf_{i\to \infty}\tilde{S}_m(\mathcal{F}_i)$$
in \cite[Proposition 6.7]{BJ20}. In our setting, it seems impossible to prove an analogous statement for \cite[Proposition 6.3]{BJ20}, and thus we will not take integral closure. Fortunately, we can prove an analogous statement for \cite[Proposition 6.7]{BJ20} using Lemma \ref{duidui} (see (\ref{hyeongsan})).

\begin{definition}
Let $x\in (X:=\Spec R,\Delta)$ be an $n$-dimensional klt singularity, fix $\nu_0\in \mathrm{Val}^*_{X,x}$ with $A_{X,\Delta}(\nu_0)<\infty$, and let $\nu\in \mathrm{Val}^*_{X,\ni x}$.

\smallskip

For a positive integer $m$ and $t\ge 0$, define
$$ \overline{\mathrm{vol}}_m\left(\nu_0;\nu\right):=\limsup_{q\to \infty}\frac{\ell\left(\frac{\mathfrak{a}_{m}(\nu)^q+\mathfrak{a}_{mq}(\nu_0)}{\mathfrak{a}_{mq}(\nu_0)}\right)}{\frac{(mq)^n}{n!}}.$$
Moreover, define
$$ \overline{S}_m(\nu_0;\nu):=\frac{n+1}{n}\frac{A_{X,\Delta}(\nu_0)}{\mathrm{vol}(\nu_0)}\int^{\infty}_{0}\overline{\mathrm{vol}}_m\left(\nu_0;\frac{\nu}{t}\right)\,\mathrm{d}t. $$
\end{definition}

\begin{theorem} \label{theorem}
Let $x\in (X=\Spec R,\Delta)$ be an $n$-dimensional klt singularity, and fix $\nu_0\in \mathrm{Val}^*_{X,x}$.

\smallskip

There exists a constant $C=C(X,\Delta,x,\nu_0)>0$ and $a_0>0$ such that, for every $m\ge a_0$ and every $\nu\in \mathrm{Val}^*_{X,\ni x}$,
$$ 0\le S(\nu_0;\nu)-\overline{S}_m(\nu_0;\nu)\le \frac{CA_{X,\Delta}(\nu)}{m}.$$
\end{theorem}

\begin{proof}
We obtain $\mathfrak{a}_{mt}(\nu)^q\subseteq \mathfrak{a}_{mqt}(\nu)$. Therefore, by definition, $\overline{S}_m(\nu_0;\nu)\le S(\nu_0;\nu)$.

\smallskip

The subadditivity (cf. \cite[Theorem]{Tak13}) implies that there exists $h\in R$ such that for every $c\ge A_{X,\Delta}(\nu)$ and every $q\ge 1$,
$$  h^{q-1}\mathcal{J}(X,\Delta,\mathfrak{a}_{\bullet}(\nu)^{qc})\subseteq \mathcal{J}(X,\Delta,\mathfrak{a}_{\bullet}(\nu)^c)^q.$$
In fact, we can take $h$ as an element of $\mathrm{Jac}_{X}\mathcal{O}_{X}(-s\Delta)$ where $\mathrm{Jac}_{X}$ is the Jacobian ideal and $s>0$ is an integer such that $s(K_X+\Delta)$ is Cartier. Moreover, \cite[Proposition 5.7]{BJ20} gives
$$ \mathcal{J}(X,\Delta,\mathfrak{a}_{\bullet}(\nu)^c)\subseteq \mathfrak{a}_{c-A_{X,\Delta}(\nu)}(\nu).$$
Hence,
\begin{equation} \label{mizmom}
h^q\mathfrak{a}_{qc}(\nu)\subseteq \mathfrak{a}_{c-A_{X,\Delta}(\nu)}(\nu)^q
\end{equation}

\smallskip

Now, define $p:=m-\nu_0(h)$, and fix $t\ge \frac{A_{X,\Delta}(\nu)}{p}$, and set $u:=\frac{pt-A_{X,\Delta}(\nu)}{m}\ge 0$. Applying (\ref{mizmom}) with $c=pt$ gives
$$ h^q\mathfrak{a}_{qpt}(\nu)\subseteq \mathfrak{a}_{pt-A_{X,\Delta}(\nu)}(\nu)^q=\mathfrak{a}_{mu}(\nu)^q. $$
Multiplication by $h^q$ therefore gives a map
$$ \frac{\mathfrak{a}_{qpt}(\nu)+\mathfrak{a}_{qp}(\nu_0)}{\mathfrak{a}_{qp}(\nu_0)}\to \frac{\mathfrak{a}_{mu}(\nu)^q+\mathfrak{a}_{mq}(\nu_0)}{\mathfrak{a}_{mq}(\nu_0)}.$$
Indeed, if $f\in \mathfrak{a}_{qp}(\nu_0)$, then $\nu_0(h^qf)\ge qm$, and therefore $h^qf\in \mathfrak{a}_{mq}(\nu_0)$, thus the map is well-defined. Moreover, if $h^qf\in \mathfrak{a}_{mq}(\nu_0)$, then
$$ \nu_0(f)\ge mq-q\nu_0(h)\ge qp,$$
and so $f\in \mathfrak{a}_{qp}(\nu_0)$. Hence, the map is injective. Thus,
$$ \ell\left(\frac{\mathfrak{a}_{mu}(\nu)^q+\mathfrak{a}_{mq}(\nu_0)}{\mathfrak{a}_{mq}(\nu_0)}\right)\ge \ell\left(\frac{\mathfrak{a}_{qpt}(\nu)+\mathfrak{a}_{qp}(\nu_0)}{\mathfrak{a}_{qp}(\nu_0)}\right).$$
By Lemma \ref{3.2}, after dividing by $\frac{(mq)^n}{n!}$ and letting $q\to \infty$, we get
\begin{equation} \label{dalseo}
\overline{\mathrm{vol}}_m\left(\nu_0;\frac{\nu}{u}\right)\ge \left(\frac{p}{m}\right)^n\mathrm{vol}\left(\nu_0;\frac{\nu}{t}\right).
\end{equation}
Since $u=\frac{pt-A_{X,\Delta}(\nu)}{m}$, integrating (\ref{dalseo}) gives
$$ \int^{\infty}_0 \overline{\mathrm{vol}}_m\left(\nu_0;\frac{\nu}{u}\right)\,\mathrm{d}u\ge \left(\frac{p}{m}\right)^{n+1}\int^{\infty}_{\frac{A_{X,\Delta}(\nu)}{p}}\mathrm{vol}\left(\nu_0;\frac{\nu}{t}\right)\,\mathrm{d}t.$$
Since $0\le \mathrm{vol}\left(\nu_0;\frac{\nu}{t}\right)\le \mathrm{vol}(\nu_0)$ by definition, we have 
$$\int^{\frac{A_{X,\Delta}(\nu)}{p}}_0\mathrm{vol}\left(\nu_0;\frac{\nu}{t}\right)\,\mathrm{d} t\le \frac{A_{X,\Delta}(\nu)}{p}\mathrm{vol}(\nu_0).$$
Therefore,
\begin{equation} \label{park}
\int^{\infty}_0 \overline{\mathrm{vol}}_m\left(\nu_0;\frac{\nu}{u}\right)\,\mathrm{d}u\ge \left(\frac{p}{m}\right)^{n+1}\left(\int^{\infty}_0\mathrm{vol}\left(\nu_0;\frac{\nu}{t}\right)\,\mathrm{d}t-\frac{A_{X,\Delta}(\nu)}{p}\mathrm{vol}(\nu_0)\right).
\end{equation}

\smallskip

By \cite[Proposition 7.45]{Xu25}, one sees that there exists $c_1>0$ such that $\nu(f)\le c_1A_{X,\Delta}(\nu)\mathrm{ord}_x(f)$, and hence
$$ \nu_0(f)\ge \nu_0(\mathfrak{m})\mathrm{ord}_x(f)\ge \frac{\nu_0(\mathfrak{m})\cdot \lambda}{c_1A_{X,\Delta}(\nu)}$$
as well as $f\in \mathfrak{a}_{\lambda}(\nu)$. Set $b_0:=\nu_0(\mathfrak{m})$. Then, $\mathfrak{a}_{\lambda}(\nu)\subseteq \mathfrak{a}_{\frac{b_0 \lambda}{c_1A_{X,\Delta}(\nu)}}(\nu_0)$. Taking $\lambda=mt$, we get

\begin{equation} \label{uchang}
\mathfrak{a}_{mt}(\nu)\subseteq \mathfrak{a}_m(\nu_0)\text{ whenever }t\ge \frac{c_1}{b_0}A_{X,\Delta}(\nu).
\end{equation}

Therefore, $\mathrm{vol}\left(\nu_0;\frac{\nu}{t}\right)=0$ for $t\ge \frac{c_1}{b_0}A_{X,\Delta}(\nu)$. Since $\mathrm{vol}\left(\nu_0;\frac{\nu}{t}\right)\le \mathrm{vol}(\nu_0)$, this gives
\begin{equation} \label{izumi}
\int^{\infty}_0 \mathrm{vol}\left(\nu_0;\frac{\nu}{t}\right)\,\mathrm{d} t\le \frac{c_1}{b_0}\mathrm{vol}(\nu_0)A_{X,\Delta}(\nu).
\end{equation}

\smallskip

Let
$$ I:=\int^{\infty}_{0} \mathrm{vol}\left(\nu_0;\frac{\nu}{t}\right)\,\mathrm{d} t,\,\,\,\,\,\,\, I_m:=\int^{\infty}_0\overline{\mathrm{vol}}_m\left(\nu_0;\frac{\nu}{t}\right)\,\mathrm{d}t. $$
Then, from (\ref{park}),
$$ I_m\ge \left(\frac{p}{m}\right)^{n+1}\left(I-\frac{A_{X,\Delta}(\nu)}{p}\mathrm{vol}(\nu_0)\right).$$
Thus,
$$ I-I_m\le \left(1-\left(\frac{p}{m}\right)^{n+1}\right)I+\left(\frac{p}{m}\right)^{n+1}\frac{A_{X,\Delta}(\nu)}{p}\mathrm{vol}(\nu_0).$$
For $m\ge 2\nu_0(h)$, we have $p=m-\nu_0(h)\ge \frac{m}{2}$, so
$$ 1-\left(\frac{p}{m}\right)^{n+1}=1-\left(1-\frac{\nu_0(h)}{m}\right)^{n+1}\le \frac{(n+1)\nu_0(h)}{m}$$
and $\frac{1}{p}\le \frac{2}{m}$. Using (\ref{izumi}), we obtain
$$ I-I_m\le \frac{(n+1)\nu_0(h)}{m}\cdot \frac{c_1}{b_0}\mathrm{vol}(\nu_0)A_{X,\Delta}(\nu)+\frac{2\mathrm{vol}(\nu_0)A_{X,\Delta}(\nu)}{m}.$$
We let $a_0:=2\nu_0(h)$.
\end{proof}

\subsection{Generic limit argument and proof of the main theorem}
In this subsection, we prove Theorem \ref{누영이} using the generic limit argument. First, we prove three lemmas that we need.

\begin{lemma} \label{듸듸<3}
Let $x\in X:=\Spec X$ be a closed point of an $n$-dimensional affine normal variety, and let $\nu_0$ be a valuation centered on $x$. For a fixed $m>0$ and a fixed ideal $I\subseteq R$, define
$$ V_m(I):=\limsup_{q\to \infty}\frac{\ell((I^q+\mathfrak{a}_{mq}(\nu_0))/\mathfrak{a}_{mq}(\nu_0))}{\frac{(mq)^n}{n!}}.$$
Then, for every $N>m$, we obtain that
$$ 0\le V_m(I+\mathfrak{a}_{N}(\nu_0))-V_m(I)\le \mathrm{vol}(\nu_0)\left(1-\left(1-\frac{m}{N}\right)^n\right).$$
\end{lemma}

\begin{proof}
One can see $V_m(I)\le V_m(I+\mathfrak{a}_{N}(\nu_0))$ by $I\subseteq I+\mathfrak{a}_N(\nu_0)$. Let us prove the reversed inequality. For each $q\ge 1$, define $r_q:=\ceil{\frac{mq}{N}}$ and $s_q:=q-r_q+1$. Since $N>m$, if $q\gg 1$, then $s_q\ge 1$, and $\frac{s_q}{q}\to 1-\frac{m}{N}$.

\smallskip

Let us claim $(I+\mathfrak{a}_{N}(\nu_0))^q\subseteq I^{s_q}+\mathfrak{a}_{mq}(\nu_0)$. Indeed, we obtain
$$ (I+\mathfrak{a}_N(\nu_0))^q=\sum^q_{r=0}I^{q-r}\mathfrak{a}_{N}(\nu_0)^r=\sum^{r_q-1}_{r=0}I^{q-r}\mathfrak{a}_{N}(\nu_0)^r+\sum^q_{r=r_q}I^{q-r}\mathfrak{a}_{N}(\nu_0)^r.$$
Moreover, if $r\ge r_q$, then $Nr\ge mq$, and therefore
$$ I^{q-r}\mathfrak{a}_{N}(\nu_0)^r\subseteq \mathfrak{a}_{Nr}(\nu_0)\subseteq \mathfrak{a}_{mq}(\nu_0).$$
If $r<r_q$, then $q-r\ge s_q$, and thus $I^{q-r}\subseteq I^{s_q}$. Thus, $I^{q-r}\mathfrak{a}_{N}(\nu_0)^r\subseteq I^{q-r}\subseteq I^{s_q}$.
Hence,
$$ \ell_R\left(\frac{(I+\mathfrak{a}_{N}(\nu_0))^q+\mathfrak{a}_{mq}(\nu_0)}{\mathfrak{a}_{mq}(\nu_0)}\right)\le \ell_R\left(\frac{I^{s_q}+\mathfrak{a}_{ms_q}(\nu_0)}{\mathfrak{a}_{ms_q}(\nu_0)}\right)+\ell_R(\mathfrak{a}_{ms_q}(\nu_0)/\mathfrak{a}_{mq}(\nu_0)).$$
Note that
$$ \limsup_{q\to \infty}\frac{\ell_R\left(\frac{I^{s_q}+\mathfrak{a}_{ms_q}(\nu_0)}{\mathfrak{a}_{ms_q}(\nu_0)}\right)}{\frac{(mq)^n}{n!}}\le \left(1-\frac{m}{N}\right)^nV_m(I).$$
It remains to estimate $\ell_R(\mathfrak{a}_{ms_q}(\nu_0)/\mathfrak{a}_{mq}(\nu_0))$. Note that
$$ \ell_R(\mathfrak{a}_{ms_q}(\nu_0)/\mathfrak{a}_{mq}(\nu_0))=\ell_R(R/\mathfrak{a}_{mq}(\nu_0))-\ell_R(R/\mathfrak{a}_{ms_q}(\nu_0)).$$
Hence, dividing both sides by $\frac{(mq)^n}{n!}$ and letting $q\to \infty$ gives
$$ \lim_{q\to \infty}\frac{\ell_R(\mathfrak{a}_{ms_q}(\nu_0)/\mathfrak{a}_{mq}(\nu_0))}{\frac{(mq)^n}{n!}}\to \left(1-\left(1-\frac{m}{N}\right)^n\right)\mathrm{vol}(\nu_0).$$
We complete the proof.
\end{proof}

\begin{lemma} \label{lemma}
Let $X$ be a noetherian scheme, $I\subseteq \Z_{>0}$ an infinite set, and $x_i\in X$ for $i\in I$ be a point.

\smallskip

Then, there exist an infinite subset $I'\subseteq I$ and an integral closed subscheme $Z\subseteq X$ such that $x_i\in Z$ for all $i\in I'$, and every proper closed subset $Y\subsetneq Z$ contains $x_i$ for only finitely many $i\in I'$.
\end{lemma}

\begin{proof}
Let us use the contradiction. Suppose the lemma is false. Then, for every infinite subset $I'\subseteq I$ and every integral closed reduced subscheme $Z\subseteq X$ such that $x_i\in Z$ for $i\in I'$, there exists a proper closed subset $Z'\subsetneq Z$ such that $\{i\in I'\mid x_i\in Z'\}$ is an infinite set. Let us apply the property to $I_0:=\Z_{>0}$ and $Z=X$. Let $Z_1$ be an irreducible component of $Z'$, and $I_1:=I'$. Applying the property to $Z_1$ and $I_1$ gives $Z_2\subsetneq Z_1\subsetneq Z$. By repeating this argument, we obtain that
$$ \cdots \subsetneq Z_n\subsetneq \cdots \subsetneq Z_2\subsetneq Z_1\subsetneq Z,$$
which contradicts the property that $X$ is noetherian.
\end{proof}

We believe that the following lemmas are well-known to experts (cf. \cite[Section 3]{Pra88}), and we include the proofs for the ease of the readers.

\begin{lemma}[{cf. \cite[Section 3]{Pra88}}] \label{duidui}
Let $X$ be a noetherian scheme, and let $\varphi:\mathcal{E}\to \mathcal{E}'$ be a morphism of vector bundles on $X$. Then, for an integer $r\ge 0$,
$$ \{x\in X\mid \mathrm{rank}_{\kappa(x)}(\varphi|_x)\le r\}$$
is a closed set.
\end{lemma}

Here, $\varphi|_x$ means the pullback of $\varphi$ along $\Spec \kappa(x)\hookrightarrow X$

\begin{proof}
The assertion is local on $X$, so we may assume $X:=\Spec A$ is affine, $\mathcal{E}\cong \mathcal{O}^{\otimes e}_X$, and $\mathcal{E}'\cong \mathcal{O}^{\otimes f}_X$. We may represent $\varphi$ by a matrix
$$ M=(m_{ij})\in \mathrm{Mat}_{f\times e}(A).$$
Let $x\in X$ correspond to a prime ideal $\mathfrak{p}_x\subseteq A$. The fiber map $\varphi|_x$ is represented by $M(x):=(m_{ij}\otimes_A \kappa(x))$.

\smallskip

Note that $\mathrm{rank}_{\kappa(x)}M(x)\le r$ if and only if every $(r+1)\times (r+1)$-minor of $M(x)$ vanishes. Equivalently, if $I_{r+1}(M)\subseteq A$ denotes the ideal generated by the $(r+1)\times (r+1)$-minors of $M$, then
$$ \mathrm{rank}_{\kappa(z)}M(x)\le r\iff I_{r+1}(M)\subseteq \mathfrak{p}_x.$$
Thus,
$$ \{x\in X\mid \mathrm{rank}_{\kappa(x)}(\varphi|_x)\le r\}=V(I_{r+1}(M)),$$
and it is closed.
\end{proof}

\begin{lemma} \label{duidui2}
Let $X$ be a noetherian scheme, and $\phi:\mathcal{F}\to \mathcal{F}'$ be a map of sheaves, where $\mathcal{F}$ is coherent and $\mathcal{F}'$ is a vector bundle.

\smallskip

Then, the set $\{x\in X\mid \phi|_x=0\}$ is closed.
\end{lemma}

\begin{proof}
Note that
$$ \dim_{\kappa}\mathrm{coker}\,\phi|_x=\dim \mathcal{F}'-\mathrm{rank}(\phi|_x),$$
so
$$ \phi|_x\iff \dim_{\kappa(x)}\mathrm{coker}\,\phi|_x\ge \dim \mathcal{F}'.$$
Moreover, the map $x\mapsto \dim_{\kappa(x)}(\mathrm{coker}\,\phi|_x)$ is upper semicontinuous by \cite[Exercise 2.5.8(a)]{Har77}. Hence, our set is closed.
\end{proof}

\begin{proof}[Proof of Theorem \ref{누영이}]
Let $N:=N(\mathbb{T})=\mathrm{Hom}(\mathbb{G}_m,\mathbb{T})$ be the co-weight lattice and $M:=N^*$ the weight lattice. We have a weight decomposition
$$ R=\bigoplus_{\alpha\in M}R_{\alpha}, \,\,\,R_0=k,\,\,\, \mathfrak{m}_x=\bigoplus_{\alpha \ne 0}R_{\alpha}.$$

\smallskip

\noindent \textbf{Step 1.} By Izumi's inequality (cf. \cite[Proposition 7.45]{Xu25}), there exists $a,b>0$ such that
\begin{equation} \label{yugang}
\mathfrak{m}^{\ceil{\frac{N}{a}}}_x\subseteq \mathfrak{a}_N(\nu_0)\subseteq \mathfrak{m}^{\ceil{\frac{N}{b}}}_x
\end{equation}
Thus, the $\mathfrak{a}_N(\nu_0)$-adic topology is the same as the $\mathfrak{m}_x$-adic topology, and
$$ \widehat{R}:=\lim_{\substack{\longleftarrow \\ N}}R/\mathfrak{a}_N(\nu_0)=\prod_{\alpha\in M}R_{\alpha}.$$
Let $P_d:=\frac{1}{d!}\Z_{\ge 0}\cap [0,d]$.

\smallskip

Let $\nu_i$ be a sequence of valuations in $\mathrm{Val}^{\mathbb{T}}_{X,\ni x}$ such that
\begin{itemize}
    \item $A_{X,\Delta}(\nu_i)=1$, and
    \item $\frac{A_{X,\Delta}(\nu_i)}{S(\nu_0;\nu_i)}\to \delta(X,\Delta,\nu_0)$.
\end{itemize}
For each $d$ and $i,p\in \Q_{\ge 0},N$, we denote by
$$ I_{i,p,N}:=\frac{\mathfrak{a}_{p}(\nu_i)+\mathfrak{a}_N(\nu_0)}{\mathfrak{a}_N(\nu_0)}.$$
Let $L_1\subseteq \Z_{>0}$ be a set such that for every $i\in L_1$, $p\in P_1$, and $1\le N\le 1$,
$$ r_{p,N}:=\dim_k I_{i,p,N} $$
is constant. Let $L_2\subseteq L_1$ be a set such that for every $i\in L_2$, $p\in P_2$, and $1\le N\le 2$, $r_{p,N}$ is constant. Repeating this, we can construct a sequence of sets $\cdots\subseteq L_3\subseteq L_2\subseteq L_1$ such that if $i\in L_d$, then $r_{p,N}$ is constant for every $p\in P_d$ and every $1\le N\le d$. Choose $i_d\in L_d$ with $i_1<i_2<i_3<\cdots$ and replace $i$ with $i_d$; we may assume that
$$ p\in P_d,\,\,\, N\le d,\,\,\, i\ge d\implies \dim I_{i,p,N}=r_{p,N}.$$

\smallskip

\noindent \textbf{Step 2.} In this step, we construct a generic limit of $\nu_i$. Define
$$ H_d:=\prod_{p\in P_d, 1\le N\le d}\mathrm{Gr}(r_{p,N},R/\mathfrak{a}_N(\nu_0))^{\mathbb{T}},$$
where $\mathrm{Gr}(r_{p,N},R/\mathfrak{a}_N(\nu_0))^{\mathbb{T}}$ denotes the moduli of $\mathbb{T}$-invariant $r$-dimensional subspaces. We consider the tautological subbundle $\mathcal{I}^{(d)}_{p,N}\subseteq (R/\mathfrak{a}_N(\nu_0))\otimes_k \mathcal{O}_{H_d}$. Let $z\in H_d$. Let us consider the following conditions:
\begin{itemize}
    \item[(1)] (decreasing property) $(\mathcal{I}^{(d)}_{q,N})_z\subseteq (\mathcal{I}^{(d)}_{p,N})_z$ for $p\le q$.
    \item[(2)] (idealhood) $(R/\mathfrak{a}_N(\nu_0))\cdot (\mathcal{I}^{(d)}_{p,N})_z\subseteq (\mathcal{I}^{(d)}_{p,N})_z$.
    \item[(3)] (multiplicativity) $(\mathcal{I}^{(d)}_{p,N})_z(\mathcal{I}^{(d)}_{q,N})_z\subseteq (\mathcal{I}^{(d)}_{p+q,N})_z$ with $p,q,p+q\in P_d$, and
    \item[(4)] (compatibility condition) $(R/\mathfrak{a}_{N'}(\nu_0)\to R/\mathfrak{a}_N(\nu_0))(\mathcal{I}^{(d)}_{p,N'})\subseteq \mathcal{I}^{(d)}_{p,N}$ for every $N\le N'\le d$.
\end{itemize}
We define
$$ J_d:=\{z\in H_d\mid (\mathcal{I}^{(d)}_{p,N})_z\text{ satisfies (1), (2), (3) and (4) for all }p,q\in P_d\}.$$
Then, $J_d\subseteq H_d$ is closed. Indeed, for (1), (2), (3), and (4), equivalent conditions are
\begin{itemize}
    \item[(1')] the composition of 
    $$\mathcal{I}^{(d)}_{q,N}\hookrightarrow (R/\mathfrak{a}_N(\nu_0))\otimes \mathcal{O}_{H_d}\to ((R/\mathfrak{a}_N(\nu_0))\otimes \mathcal{O}_{H_d})/\mathcal{I}^{(d)}_{p,N}$$
    is a zero map at $z$,
    \item[(2')] the map
    $$(R/\mathfrak{a}_N(\nu_0))\otimes \mathcal{I}^{(d)}_{p,N}\to (R/\mathfrak{a}_N(\nu_0)\otimes \mathcal{O}_{H_d})/\mathcal{I}^{(d)}_{p,N},$$ 
    $a\otimes f\mapsto af \,\,\,\mathrm{mod}\,\mathcal{I}^{(d)}_{p,N}$ is a zero map at $z$,
    \item[(3')] the map 
    $$\mathcal{I}^{(d)}_{p,N}\otimes \mathcal{I}^{(d)}_{q,N}\to (R/\mathfrak{a}_N(\nu_0)\otimes \mathcal{O}_{H_d})/\mathcal{I}^{(d)}_{p+q,N}$$
    is a zero map at $z$, and
    \item[(4')] the composition of
    $$ \mathcal{I}^{(d)}_{p,N'}\hookrightarrow (R/\mathfrak{a}_{N'}(\nu_0))\otimes \mathcal{O}_{H_d}\to (R/\mathfrak{a}_N(\nu_0))\otimes \mathcal{O}_{H_d}\to ((R/\mathfrak{a}_N(\nu_0))\otimes \mathcal{O}_{H_d})/\mathcal{I}^{(d)}_{p,N}$$
    is a zero map at $z$.
\end{itemize}
Note that since $\mathcal{I}^{(d)}_{p,N}$ is a tautological subbundle of a Grassmannian, $((R/\mathfrak{a}_N(\nu_0))\otimes \mathcal{O}_{H_d})/\mathcal{I}^{(d)}_{p,N}$ is thus a vector bundle on $H_d$ (cf. \cite[Proposition 8.10 (iv) $\implies$ (ii)]{GW20}). Hence, 
$$ ((R/\mathfrak{a}_N(\nu_0))\otimes \mathcal{O}_{H_d})/\mathcal{I}^{(d)}_{p,N}|_z=((R/\mathfrak{a}_N(\nu_0)\otimes_k \kappa(z))/(\mathcal{I}^{(d)}_{p,N}|_z).$$
We can apply Lemma \ref{duidui2} to deduce the closedness of the conditions.

\smallskip

We denote by $\mathcal{I}^{(d)}_{p,N}$ the restriction of $\mathcal{I}^{(d)}_{p,N}$ to $J_d$ as an abuse of notation. Define the transition map by $\phi^{(d)}_{p,N}:\mathcal{I}^{(d)}_{p,N+1}\to \mathcal{I}^{(d)}_{p,N}$, and we define
$$ J^{\circ}_d:=\bigcap_{\substack{p\in P_d \\ 1\le N<d}}\{z\in J_d\mid (\phi^{(d)}_{p,N})_z\text{ is surjective}\}.$$
Then, we know that $J^{\circ}_d\subseteq J_d$ is open. Indeed, if we let $\mathcal{C}^{(d)}_{p,N}$ be the cokernel of $\phi^{(d)}_{p,N}$, then $(\phi^{(d)}_{p,N})_z$ is surjective if and only if $z\notin \Supp \mathcal{C}^{(d)}_{p,N}$.

\smallskip

We combine the idea of constructing $k_0$ from \cite[Proof of Theorem 4.1]{Tot12} and the simultaneous descent of the countable collection of valuative ideals argument from \cite[Proof of Proposition 3.9]{LX20}. Let $k_0\subseteq k$ be an algebraically closed, countable field such that
\begin{itemize}
    \item there exists a log Fano cone singularity $x_0\in (X_0:=\Spec R_0,\Delta_0;\mathbb{T}_0)$ over $k_0$ such that $(X_0,\Delta_0;\mathbb{T}_0)\times_{k_0}k=(X,\Delta;\mathbb{T})$,
    \item There exist valuations $\nu_{i0}\in \mathrm{Val}_{X_0}$ and $z_{i,d}\in \left(I_{i,p,N,0}\right)_{p\in P_d,1\le N\le d}\in J^{\circ}_{d0}(k_0)$ such that $\mathfrak{a}_{p}(\nu_{i0})\otimes_{k_0} k=\mathfrak{a}_{p}(\nu_i)$ and $I_{i,p,N,0}\otimes_{k_0}k=I_{i,p,N}$, and
    \item There exist $\nu_{00}\in \mathrm{Val}_{X_0}$ such that $\mathfrak{a}_p(\nu_{00})\otimes_{k_0}k=\mathfrak{a}_p(\nu_0)$.
\end{itemize}
Note that we can construct $J_{d0}$, $J^{\circ}_{d0}$, $\phi_{d',d,0}:J_{d'0}\to J_{d0}$, and $\mathcal{I}^{(d)}_{p,N,0}\subseteq R_{0}/\mathfrak{a}_N(\nu_{00})\otimes_{k_0}\mathcal{O}_{J_{d0}}$ over $k_0$ similarly to $J_{d}$, $J^{\circ}_{d}$, $\phi_{d',d}:J_{d'}\to J_{d}$, and $\mathcal{I}^{(d)}_{p,N}$.

\smallskip

We may choose infinite subsets $\Z_{>0}\supseteq I_0\supseteq I_1\supseteq \cdots$ such that for each $d$, $I_d\cap \{1,2,\cdots,d\}=\varnothing$ and the closed set
$$ Z_{d0}:=\overline{\{z_{i,d}\mid i\in I_d\}}\subseteq J_{d0} $$
satisfies the property:
\begin{equation} \label{namseong-li}
\begin{aligned}\text{ If }Y\subsetneq Z_{d0}&\text{ is a closed set, }\\ &\text{there are only finitely many }i\in I_d\text{ such that }\pi_d(z_i)\in Y.
\end{aligned}
\end{equation}
This is possible by Lemma \ref{lemma} For the transition morphism $\varphi_{d+1,d,0}:J_{(d+1)0}\to J_{d0}$, we have that $\varphi_{d+1,d,0}(z_{i,d+1})=z_{i,d}$. Moreover, $I_{d+1}\subseteq I_d$. Thus, $\varphi_{d+1,d,0}(Z_{(d+1)0})\subseteq Z_{d0}$. Let us assume that the closure of the image $Y$ is a proper closed subset of $Z_{d0}$. Then, for every $i\in I_{d+1}$, $z_{i,d}\in Y$, contradicting to (\ref{namseong-li}). Thus, $Z_{d+1}\to Z_d$ is dominant. Moreover, since every $z_{i,d}$ is in $J^{\circ}_{d0}$, $J^{\circ}_{d0}\cap Z_d$ is a nonempty open subset of $Z_d$.

\smallskip

Let $\eta_d$ be the generic point of $Z_{d0}$, and let $K_d:=\kappa(\eta_d)$. Then, we obtain $K_d\subseteq K_{d+1}$. We denote the generic fiber of $\mathcal{I}^{(d)}_{p,N,0}$ by $I^{(d)}_{p,N,0}\subseteq R_{0}/\mathfrak{a}_N(\nu_{00})\otimes_{k_0} K_d$. Then, we deduce that
$$ I^{(d+1)}_{p,N,0}=I^{(d)}_{p,N,0}\otimes_{K_d}K_{d+1}$$
for $p\in P_d,1\le N\le d$. Let $K_{\infty}:=\bigcup_d K_d$. Choose $d$ such that $p\in P_d$ and $N\le d$, and let
$$ I^{(\infty)}_{p,N,0}:=I^{(d)}_{p,N,0}\otimes_{K_d} K_{\infty}\subseteq (R_{0}/\mathfrak{a}_N(\nu_{00}))\otimes_{k_0} K_{\infty}.$$
Choose any embedding $K_{\infty}\hookrightarrow k$ and define
$$ I^{(\infty)}_{p,N}:=I^{(\infty)}_{p,N,0}\otimes_{K_{\infty}}k\subseteq R_{0}/\mathfrak{a}_N(\nu_{00})\otimes_{k_0}k=R/\mathfrak{a}_N(\nu_{0}).$$
Define
$$ \widehat{\mathcal{F}_{\mathbb{Q}}}^p:=\lim_{\substack{\longleftarrow \\  N}}I^{(\infty)}_{p,N}\subseteq \widehat{R}.$$
By construction, every $f\in \widehat{\mathcal{F}_{\mathbb{Q}}}^p$ can be written as
$$ f=\sum_{\alpha}f_{\alpha},\,\,\,f_{\alpha}\in R_{\alpha}$$
with $f_{\alpha}\in \widehat{\mathcal{F}_{\mathbb{Q}}}^p$; thus,
$$ \mathcal{F}^p_{\mathbb{Q}}:=\bigoplus_{\alpha}(\widehat{\mathcal{F}_{\mathbb{Q}}}^p\cap R_{\alpha})\subseteq R$$
is a $\mathbb{T}$-invariant ideal. Then, we can see $\widehat{\mathcal{F}}^p_{\mathbb{Q}}=\mathcal{F}^p_{\mathbb{Q}}\widehat{R}$. We define that for a positive real number $t$,
$$ \mathcal{F}^t:=\bigcap_{\substack{p<t \\ p\in \Q}}\mathcal{F}^p_{\mathbb{Q}}.$$
We see that $\mathcal{F}^{\bullet}$ is a $\mathbb{T}$-invariant filtration on $R$.

\smallskip

\noindent \textbf{Step 3.} The goal of this step is to prove (\ref{D_m}). Fix $p,N$ and choose $d\ge N$ such that $p\in P_d$. Let $q$ be an integer such that $mq\ge N$. Then, there exists $\pi_{mq,N}:R_{0}/\mathfrak{a}_{mq}(\nu_{00})\to R_{0}/\mathfrak{a}_N(\nu_{00})$, and define
$$ \tilde{\mathcal{I}}^{(d)}_{p,N,mq}:=(\pi_{mq,N}\times \mathrm{id}_{Z_{d0}})^{-1}(\mathcal{I}^{(d)}_{p,N,0})\subseteq R_{0}/\mathfrak{a}_{mq}(\nu_{00})\otimes_{k_0}\mathcal{O}_{Z_{d0}}.$$
Note that
\begin{equation} \label{vector}
\tilde{\mathcal{I}}^{(d)}_{p,N,mq}\text{ is a vector bundle on }Z_{d0}.
\end{equation}
Indeed, there exists an exact sequence of $\mathcal{O}_{Z_{d0}}$-modules
\begin{equation} \label{ex}
0\to (\mathfrak{a}_N(\nu_{00})/\mathfrak{a}_{mq}(\nu_{00}))\otimes \mathcal{O}_{Z_{d0}}\to \tilde{\mathcal{I}}^{(d)}_{p,N,mq}\to \mathcal{I}^{(d)}_{p,N,0}\to 0,
\end{equation}
(\ref{ex}) splits as $\mathcal{O}_{Z_{d0}}$-modules by a $k_0$-vector space section $R_{0}/\mathfrak{a}_{N}(\nu_{00})\to R_{0}/\mathfrak{a}_{mq}(\nu_{00})$, and $\mathcal{I}^{(d)}_{p,N,0}$ is a direct sum of finitely many tautological subbundles.

\smallskip

Let us define
$$ J^{(d)}_{p,N,0}:=\left(R_{0}\otimes_{k_0}K_d\to \frac{R_{0}\otimes_{k_0}K_d}{\mathfrak{a}_N(\nu_{00})\otimes_{k_0}K_d}\right)^{-1}\left(I^{(d)}_{p,N,0}\right)\subseteq R_{0}\otimes_{k_0}K_d,$$
and $J^{(\infty)}_{p,N,0}\subseteq R_{0}\otimes_{k_0}K_{\infty}$ similarly. Consider
\begin{equation}\label{jukdodong}
\tilde{\mathcal{I}}^{(d)}_{p,N,mq}|_z=((\pi_{mq,N}\times \mathrm{id}_{Z_{d0}})|_z)^{-1}\left(\mathcal{I}^{(d)}_{p,N,0}|_z\right).
\end{equation}
for every $z\in Z_{d0}$. Note that (\ref{jukdodong}) comes from the fact that $((R_0/\mathfrak{a}_N(\nu_{00}))\otimes \mathcal{O}_{Z_{d0}})/\mathcal{I}^{(d)}_{p,N,0}$ is a vector bundle. Using (\ref{jukdodong}), we obtain
\begin{equation}
\begin{aligned} \label{daeidong}
\tilde{\mathcal{I}}^{(d)}_{p,N,mq}|_{z_{i,d}}&=\pi^{-1}_{mq,N}\left(\frac{\mathfrak{a}_p(\nu_{i0})+\mathfrak{a}_N(\nu_{00})}{\mathfrak{a}_N(\nu_{00})}\right)
\\ &=\left\{f+\mathfrak{a}_{mq}(\nu_{00})\mid f+\mathfrak{a}_N(\nu_{00})\in \frac{\mathfrak{a}_p(\nu_{i0})+\mathfrak{a}_N(\nu_{00})}{\mathfrak{a}_N(\nu_{00})}\right\}
\\ &=\frac{\mathfrak{a}_p(\nu_{i0})+\mathfrak{a}_N(\nu_{00})+\mathfrak{a}_{mq}(\nu_{00})}{\mathfrak{a}_{mq}(\nu_{00})}
\\ &=\frac{\mathfrak{a}_p(\nu_{i0})+\mathfrak{a}_N(\nu_{00})}{\mathfrak{a}_{mq}(\nu_{00})},
\end{aligned}
\end{equation}
and
\begin{equation} \label{yanghakdong}
\tilde{\mathcal{I}}^{(d)}_{p,N,mq}|_{\eta_d}=\frac{J^{(d)}_{p,N,0}+\mathfrak{a}_{mq}(\nu_{00})\otimes_{k_0}K_d}{\mathfrak{a}_{mq}(\nu_{00})\otimes_{k_0}K_d}
\end{equation}
Let us define the multiplicative map of $\mathcal{O}_{Z_{d0}}$-modules
$$ \mu_{p,N,q}:(\tilde{\mathcal{I}}^{(d)}_{p,N,mq})^{\otimes q}\to (R_{0}/\mathfrak{a}_{mq}(\nu_{00}))\otimes \mathcal{O}_{Z_{d0}}.$$
Then, by (\ref{daeidong}) and (\ref{yanghakdong}), we obtain
\begin{equation} \label{seonggok-li}
\begin{aligned}
&\mathrm{Im}((\mu_{p,N,q})|_{z_{i,d}})=\frac{(\mathfrak{a}_{p}(\nu_{i0})+\mathfrak{a}_N(\nu_{00}))^q+\mathfrak{a}_{mq}(\nu_{00})}{\mathfrak{a}_{mq}(\nu_{00})}
\\ & \mathrm{Im}((\mu_{p,N,q})|_{\eta_d})=\frac{\left(J^{(d)}_{p,N,0}\right)^q+\mathfrak{a}_{mq}(\nu_{00})\otimes_{k_0}K_{d}}{\mathfrak{a}_{mq}(\nu_{00})\otimes_{k_0}K_{d}}.
\end{aligned}
\end{equation}
by the fact that $(I/J)^q=\frac{I^q+J}{J}$ for two ideals $J\subseteq I\subseteq R_{0}$. Hence, after replacing $\{\nu_i\}$ once and for all with the diagonal subsequence $\{\nu_{i_j}\}$ and $i_j\in I_j$, the following inequality holds for every $q$ with $mq\ge N$ and $i\ge d$:
$$ 
\begin{aligned}
\ell_R&\left(\frac{(\mathcal{F}^p_{\mathbb{Q}}+\mathfrak{a}_N(\nu_{0}))^q+\mathfrak{a}_{mq}(\nu_{0})}{\mathfrak{a}_{mq}(\nu_{0})}\right) & \\ &=\ell_{R_0\otimes_{k_0}K_{\infty}}\left(\frac{\left(J^{(\infty)}_{p,N,0}\right)^q+\mathfrak{a}_{mq}(\nu_{00})\otimes_{k_0}K_{\infty}}{\mathfrak{a}_{mq}(\nu_{00})\otimes_{k_0}K_{\infty}}\right) & (1)
\\ &=\ell_{R_0\otimes_{k_0}K_d}\left(\frac{\left(J^{(d)}_{p,N,0}\right)^q+\mathfrak{a}_{mq}(\nu_{00})\otimes_{k_0}K_{d}}{\mathfrak{a}_{mq}(\nu_{00})\otimes_{k_0}K_{d}}\right) &
\\ &\ge \ell_{R_0}\left(\frac{(\mathfrak{a}_{p}(\nu_{i0})+\mathfrak{a}_N(\nu_{00}))^q+\mathfrak{a}_{mq}(\nu_{00})}{\mathfrak{a}_{mq}(\nu_{00})}\right) & (2)
\\ &=\ell_R\left(\frac{(\mathfrak{a}_{p}(\nu_{i})+\mathfrak{a}_N(\nu_{0}))^q+\mathfrak{a}_{mq}(\nu_{0})}{\mathfrak{a}_{mq}(\nu_{0})}\right), &
\end{aligned}
$$
where
\begin{itemize}
\item (1) follows from $\frac{\mathcal{F}^{p}_{\Q}+\mathfrak{a}_N(\nu_0)}{\mathfrak{a}_N(\nu_0)}=I^{(\infty)}_{p,N}$, and then $J^{(\infty)}_{p,N,0}\otimes_{K_{\infty}}k=\mathcal{F}^p_{\mathbb{Q}}+\mathfrak{a}_N(\nu_0)$, and
    \item (2) is due to
    $$ \dim_{K_d}\mathrm{Im}\,(\mu_{p,N,q})|_{\eta_d}\ge \dim_{\kappa(z_{i,d})}\mathrm{Im}\,(\mu_{p,N,q})|_{z_{i,d}}$$
    (cf. (\ref{vector}), (\ref{seonggok-li}), and Lemma \ref{duidui}).
\end{itemize}
Hence, by dividing both sides by $\frac{(mq)^n}{n!}$ and taking $\limsup$ of $q\to \infty$, we get
$$ V_m(\mathcal{F}^p+\mathfrak{a}_N(\nu_0))\ge V_m(\mathcal{F}^p_{\mathbb{Q}} +\mathfrak{a}_N(\nu_0))\ge V_m(\mathfrak{a}_{p}(\nu_i)+\mathfrak{a}_N(\nu_0))$$
(where we used the same notation as in Lemma \ref{듸듸<3}). Hence,
$$ V_m(\mathcal{F}^p+\mathfrak{a}_N(\nu_0))\ge \limsup_{i\to \infty}V_m(\mathfrak{a}_{p}(\nu_i)+\mathfrak{a}_N(\nu_0)).$$
Lemma \ref{듸듸<3} implies
\begin{equation} \label{D_m}
\begin{aligned}V_m(\mathcal{F}^p)&=\lim_{N\to \infty}V_m(\mathcal{F}^p+\mathfrak{a}_N(\nu_0))
\\ &\ge \lim_{N\to \infty}\limsup_{i\to \infty}V_m(\mathfrak{a}_{p}(\nu_i)+\mathfrak{a}_N(\nu_0))
\\ &\ge \lim_{N\to \infty}\limsup_{i\to \infty}V_m(\mathfrak{a}_{p}(\nu_i))
\\ &=\limsup_{i\to \infty}V_m(\mathfrak{a}_{p}(\nu_i)).
\end{aligned}
\end{equation}
Since $p\mapsto V_m(\mathcal{F}^p)$ and
$$ p\mapsto \limsup_{i\to \infty}V_m(\mathfrak{a}_{p}(\nu_i)) $$
are nonincreasing, they are continuous almost everywhere, and thus we can easily extend (\ref{D_m}) to almost all real $p$. 

\smallskip

\noindent \textbf{Step 4.} Let us use Xu's quasi-monomiality theorem to conclude the proof. There exists a $\mathbb{T}$-invariant quasi-monomial valuation $\nu'\in \mathrm{Val}^{\mathbb{T}}_X$ that computes $\mathrm{lct}_x(X,\Delta,\mathcal{F}^{\bullet})$ by Lemma \ref{hevensteelyard}. We may assume $\nu'(\mathcal{F}^{\bullet})=1$. Let us define
$$ \overline{S}_m(\nu_0;\mathcal{F}^\bullet):=\frac{n+1}{n}\frac{A_{X,\Delta}(\nu_0)}{\mathrm{vol}(\nu_0)}\int^{\infty}_{0}
V_m(\mathcal{F}^{mt})\,\mathrm{d}t
.$$
Then, for every rational $p$, $\mathcal{F}^p\subseteq \mathfrak{a}_p(\nu')$. Moreover, we can see that
\begin{equation} \label{pohanghang}
\overline{S}_m(\nu_0;\mathcal{F}^{\bullet})\le \overline{S}_m(\nu_0;\nu').
\end{equation}
Furthermore, by (\ref{D_m}) and the reverse Fatou's lemma (see also (\ref{uchang})),
\begin{equation} \label{hyeongsan}
\begin{aligned}
\overline{S}_m(\nu_0;\mathcal{F}^{\bullet})&=\frac{n+1}{n}\frac{A_{X,\Delta}(\nu_0)}{\mathrm{vol}(\nu_0)}\int^{\infty}_{0}V_m(\mathcal{F}^{mt})\,\mathrm{d}t
\\ &\ge \frac{n+1}{n}\frac{A_{X,\Delta}(\nu_0)}{\mathrm{vol}(\nu_0)}\int^{\infty}_{0}\limsup_{i\to \infty}V_m(\mathfrak{a}_{mt}(\nu_i))\,\mathrm{d}t
\\ &\ge \limsup_{i\to \infty}\frac{n+1}{n}\frac{A_{X,\Delta}(\nu_0)}{\mathrm{vol}(\nu_0)}\int^{\infty}_0V_m(\mathfrak{a}_{mt}(\nu_i))\,\mathrm{d}t
\\ &=\limsup_{i\to \infty} \overline{S}_m(\nu_0;\nu_i).
\end{aligned}
\end{equation}
By (\ref{pohanghang}), (\ref{hyeongsan}), and the uniform Fujita approximation Theorem \ref{theorem},
$$ S(\nu_0;\nu')\ge \limsup_{i\to \infty} S(\nu_0;\nu_i)-\frac{C}{m}=\frac{1}{\delta(X,\Delta;\nu_0)}-\frac{C}{m}.$$
Hence, by letting $m$ approach infinity, $S(\nu_0;\nu')\ge \frac{1}{\delta(X,\Delta;\nu_0)}$. Let us prove $\mathrm{lct}_x(X,\Delta,\mathcal{F}^{\bullet})\le 1$. Granted this claim, we obtain that $A_{X,\Delta}(\nu')\le 1$, and thus $\frac{A_{X,\Delta}(\nu')}{S(\nu_0,\nu')}\le \delta(X,\Delta;\nu_0)$, which completes the proof of the theorem.

\smallskip

Note that
$$ \mathrm{lct}_x(X,\Delta,\mathfrak{a}_{p}(\nu_i))\le \frac{A_{X,\Delta}(\nu_i)}{\nu_i(\mathfrak{a}_{p}(\nu_i))}\le \frac{1}{p},$$
and thus by (\ref{yugang}) and \cite[Lemma 3.14]{BLXZ25},
\begin{equation} \label{jungdan}
\begin{aligned}
\mathrm{lct}_x&(X,\Delta,\mathfrak{a}_{p}(\nu_i)+\mathfrak{a}_N(\nu_0))
\\ &\le \mathrm{lct}_x(X,\Delta,\mathfrak{a}_{p}(\nu_i))+\mathrm{lct}_x(X,\Delta,\mathfrak{a}_N(\nu_0))
\\ &\le \frac{1}{p}+\frac{bC_x}{N},
\end{aligned}
\end{equation}
where $C_x:=\mathrm{lct}_x(X,\Delta,\mathfrak{m}_x)$. Using \cite[Lemma 2.14 iii)]{dFEM11} twice, the lower-semicontinuity of log canonical thresholds (cf. \cite[Proposition 10(2)]{BL21}) and (\ref{jungdan}) to 
$$(X_{0}\times Z_{d0},\Delta_{0}\times Z_{d0},\mathcal{J}^{(d)}_{r,N,0})\to Z_{d0},$$ where
$$ \mathcal{J}^{(d)}_{r,N,0}:=\ker\left(R_{0}\otimes_{k_0}\mathcal{O}_{Z_{d0}}\to ((R_{0}/\mathfrak{a}_N(\nu_{00}))\otimes_{k_0} \mathcal{O}_{Z_{d0}})/(\mathcal{I}^{(d)}_{r,N,0})\right),$$
allows us to obtain
$$\mathrm{lct}_x(X,\Delta,\mathcal{F}^{p}+\mathfrak{a}_N(\nu_0)) \le \mathrm{lct}_x(X,\Delta,\mathcal{F}^{p'}_{\mathbb{Q}}+\mathfrak{a}_N(\nu_0))\le \frac{1}{p'}+\frac{bC_x}{N}.$$
for every $p'<p$ (after shrinking, we may assume that $Z_{d0}$ are normal). Thus,
$$ \mathrm{lct}_x(X,\Delta,\mathcal{F}^{\bullet})=\sup_{p\ge 1}p\cdot \mathrm{lct}_x(X,\Delta,\mathcal{F}^p)\le 1.$$
We complete the proof.
\end{proof}

\begin{remark}
One may think we can prove Theorem \ref{누영이} in the countable setting using the argument of \cite{KL26}. A difficulty in importing the argument of \cite{KL26} to extend Theorem \ref{누영이} is that we do not know how to prove an analog of \cite[Lemma 3.1]{KL26} in our setting. A main problem is that we do not know how to ensure the flatness of $\frac{\mathfrak{a}_{r}(\nu)+\mathfrak{a}_r(\nu_0)}{\mathfrak{a}_r(\nu_0)}$ simultaneously for $r$.

\smallskip

Note that our generic limit argument works for every klt pair $(X,\Delta)$ of a \emph{complete} local scheme (with a natural definition of local $S$-invariant) without being a log Fano cone singularity. As explained in the introduction, our generic limit argument stops working if we omit the completeness condition (cf. Example \ref{chatchat}).
\end{remark}

\section{Questions} \label{4}
Let us pose a question.

\begin{question} \label{chatchatchat}
Let $x\in (X,\Delta)$ be a klt pair of a local scheme essentially of finite type over an algebraically closed field of characteristic $0$, $\nu_0$ a non-zero valuation over $X$ whose center is $x$, and let us define
$$ \delta(X,\Delta;\nu_0):=\inf_{\nu\in \mathrm{Val}^*_{X,\ni x}}\frac{A_{X,\Delta}(\nu)}{S(\nu_0;\nu)}.$$
Can we find any minimizer of $\delta(X,\Delta;\nu_0)$?
\end{question}

We can conclude that the answer is \emph{no}.

\begin{example} \label{chatchat}
Let us denote the local $S$-invariant of $(X,\Delta)$ by $S_{X,\Delta}(\nu_0;\nu)$.

\smallskip

Let $A:=k[s,t]_{(s,t)}$, $X:=\Spec A$, $f=t^2-s^2(1+s)$, $C=(f=0)$, $\Delta:=\frac{1}{2}C$. Then, $C$ has a node at $x$. Indeed, let $B=(A[u]/(u^2-(1+s)))_{(s,t,u-1)}$. Then,
$$ \pi:(Y,\tilde{x}):=(\Spec S,(s,t,u-1))\to (X,x)$$
is an étale morphism.

\smallskip

Now, set
$$ s':=t-su, \,\,\, t':=t+su.$$
Then, $s',t'$ are regular parameters of $\mathcal{O}_{Y,\tilde{x}}$, and $f=s't'$. Thus, if we let $\Delta_Y:=\pi^*\Delta$, then $(Y,\Delta_Y)$ is klt, and thus $x\in (X,\Delta)$ is a klt germ (cf. \cite[Proposition 2.24]{HLQ23}).

\smallskip

Let us consider the monomial valuation $\nu'_0:=\mathrm{wt}_{1,2}$, and let $\nu_0:=\nu'_0|_{K(X)}$. Then, the log discrepancy is $A_{Y,\Delta_Y}(\nu'_0)=\frac{3}{2}$, and $A_{X,\Delta}(\nu_0)=\frac{3}{2}$.

Note that for $B:=\mathcal{O}_{Y,\tilde{x}}$,
$$ \widehat{B}=k[[s',t']],$$
and
$$ \mathrm{wt}_{(a,b)}\left(\sum_{i,j}c_{ij}(s')^i(t')^j\right)=\min\{ai+bj\mid c_{ij}\ne 0\}.$$
Hence,
$$ \mathfrak{a}_{r}(\nu'_0)\widehat{B}=((s')^i(t')^j\mid i+2j\ge r)$$
and
$$ \mathfrak{a}_{r\lambda}(w)\widehat{B}=((s')^i(t')^j\mid ai+bj\ge r\lambda). $$
For a positive integer $m$, let $N_m:=\ell_B(B/\mathfrak{a}_m(\nu'_0))$. Then,
$$ N_m=|\{(i,j)\in \Z^2_{\ge 0}\mid i+2j<m\}|,$$
and hence $N_m=\frac{m^2}{4}+O(m)$. Thus, $\mathrm{vol}(\nu_0)=\frac{1}{2}$.

\smallskip

Let
$$ L_m(\lambda):=\ell_B\left(\frac{\mathfrak{a}_{m\lambda}(w)+a_m(\nu'_0)}{\mathfrak{a}_m(\nu'_0)}\right)$$
for $w=\mathrm{wt}_{(a,b)}$. Then,
$$ L_m(\lambda)=|\{(i,j)\in \Z^2_{\ge 0}\mid i+2j<m,\,ai+bj\ge m\lambda\}|.$$
Let
$$P:=\{(x,y)\in \R^2_{\ge 0}\mid x+2y<1\},$$
and
$$ P_{\lambda}:=\{(x,y)\in P\mid ax+by\ge \lambda\}.$$
Then, $\frac{L_m(\lambda)}{m^2}\to \mathrm{Area}(P_{\lambda})$. Hence,
$$ \mathrm{vol}\left(\nu_0;\frac{w}{\lambda}\right)=2\mathrm{Area}(P_{\lambda}).$$

\smallskip

One can compute
$$ \begin{aligned}
\int^{\infty}_0L_m(\lambda)\,\mathrm{d}\lambda&=\int^{\infty}_0\sum_{\substack{i,j\ge 0 \\ i+2j<m}}1_{\{ai+bj\ge m\lambda\}}\,\mathrm{d}\lambda
\\ &=\sum_{\substack{i,j\ge 0 \\ i+2j<m}}\int^{\infty}_01_{\left\{0\le \lambda\le \frac{ai+bj}{m}\right\}}\,\mathrm{d}\lambda
\\ &=\frac{1}{m}\sum_{\substack{i,j\ge 0 \\ i+2j<m}}(ai+bj).
\end{aligned}$$
Hence,
$$ \int^{\infty}_{0}\frac{L_m(\lambda)}{\frac{m^2}{2}}\,\mathrm{d}\lambda=\frac{2}{m^3}\sum_{i+2j<m}(ai+bj).$$
The function inside the integral is $L^1$. Hence, dominated convergence theorem gives
$$
\begin{aligned}
\int^{\infty}_0\mathrm{vol}\left(\nu_0;\frac{w}{\lambda}\right)\,\mathrm{d}\lambda&=\lim_{m\to \infty}\int^{\infty}_0\frac{L_m(\lambda)}{\frac{m^2}{2}}\,\mathrm{d}\lambda
\\ &=\lim_{m\to \infty}\frac{2}{m^3}\sum_{i+2j<m}(ai+bj).
\end{aligned}
$$
Moreover, $ \mathrm{vol}(\nu_0)=\lim_{m\to \infty}\frac{N_m}{\frac{m^2}{2}}.$ Thus,
$$
\frac{1}{\mathrm{vol}(\nu_0)}\int^{\infty}_0\mathrm{vol}\left(\nu_0;\frac{w}{\lambda}\right)\,\mathrm{d}\lambda=\lim_{m\to \infty}\frac{1}{mN_m}\sum_{i+2j<m}(ai+bj).
$$
Set
$$A_m:=\sum_{\substack{i,j\ge0\\i+2j<m}}i,\,\,\,
B_m:=\sum_{\substack{i,j\ge0\\i+2j<m}}j.$$
Then,
$$\sum_{i+2j<m}(ai+bj)=aA_m+bB_m.$$
Moreover, we can compute
$$ A_m=\sum^{\floor{\frac{m-1}{2}}}_{j=0}\frac{(m-2j-1)(m-2j)}{2}=\frac{m^3}{12}+O(m^2).$$
Similarly, $B_m=\frac{m^3}{24}+O(m^2)$. Thus,
$$ \lim_{m\to \infty}\frac{1}{mN_m}\sum_{\substack{i,j\ge 0 \\ i+2j<m}}(ai+bj)=\frac{a}{3}+\frac{b}{6}.$$
Finally, we obtain that
\begin{equation} \label{gpt1}
S_{Y,\Delta_Y}(\mathrm{wt}_{(1,2)},\mathrm{wt}_{(a,b)})=\frac{3}{4}\left(a+\frac{b}{2}\right)
\end{equation}
for every $a,b\ge 0$.

\smallskip

Let $w_t:=\mathrm{wt}_{(1,t)}$. Then, $A_{Y,\Delta_Y}(w_t)=\frac{1+t}{2}$ and $S(\mathrm{wt}_{(1,2)};w_t)=\frac{3}{4}\left(1+\frac{t}{2}\right)$. Therefore,
$$ \frac{A_{Y,\Delta_Y}(w_t)}{S(\mathrm{wt}_{(1,2)};w_t)}\to \frac{2}{3}\text{ as }t\to 0^+.$$
Let $q_t:=w_t|_{K(X)}$. By Lemma \ref{HLQ}, we obtain that
\begin{equation} \label{gpt2}
A_{X,\Delta}(q_t)=A_{Y,\Delta_Y}(w_t),\,\,\, S(\nu_0;q_t)=S(\mathrm{wt}_{(1,2)};w_t).
\end{equation}
Hence, $\delta(X,\Delta;\nu_0)\le \frac{2}{3}$.

\smallskip

Let $S:=(s'=0), T:=(t'=0)$, and define
$$ \Delta^*:=\frac{1}{2}S+\frac{3}{4}T.$$
Then, $A_{Y,\Delta^*}(\mathrm{wt}_{(1,2)})=1$. Think about the Kummer cover $\phi_{\A}:\A^2_{s'',t''}\to \A^2_{s',t'}$, $s'=(s'')^2,t'=(t'')^4$. Let
$$ B':=(B[s'',t'']/((s'')^2-s',(t'')^4-t')_{\mathfrak{m}'},$$
where $\mathfrak{m}'=(s'',t'')\subseteq B[s'',t'']/((s'')^2-s',(t'')^4-t')$. Let $(Y',y'):=(\Spec B',\mathfrak{m}')$. Then, we can consider $\phi:(Y',y')\to (Y,\tilde{x})$. It has degree $8$ as a finite Galois cover, $\mathrm{Ram}(\phi)=(s''=0)+3(t''=0)$, and $\phi^*\Delta^*=(s''=0)+3(t''=0)=\mathrm{Ram}(\phi)$. Hence,
$$ K_{Y'}=\phi^*(K_Y+\Delta^*).$$
We can compute $\widehat{\mathrm{vol}}(\tilde{x},Y,\Delta^*)=\frac{1}{2}$ by \cite[Theorem 1.6]{LX19} and \cite[Theorem 1.3]{XZ20}. Moreover, $\mathrm{vol}(\mathrm{wt}_{(1,2)})=\frac{1}{2}$. Thus, $\mathrm{wt}_{(1,2)}$ is the normalized-volume minimizer and K-semistable valuation (cf. \cite[Theorem 3.10]{XZ20}). Therefore, for every $w\in \mathrm{Val}^*_{Y,\tilde{x}}$,
$$ A_{Y,\Delta^*}(w)\ge S_{Y,\Delta^*}(\mathrm{wt}_{(1,2)},w).$$
Moreover,
$$ S_{Y,\Delta^*}(\mathrm{wt}_{(1,2)};w)=\frac{A_{Y,\Delta^*}(\mathrm{wt}_{(1,2)})}{A_{Y,\Delta_Y}(\mathrm{wt}_{(1,2)})}S_{Y,\Delta_Y}(\mathrm{wt}_{(1,2)};w)=\frac{2}{3}S_{Y,\Delta_Y}(\mathrm{wt}_{(1,2)};w).$$
Furthermore,
$$ A_{Y,\Delta_Y}(w)=A_{Y,\Delta^*}(w)+\frac{1}{4}w(t').$$
If $w$ is centered on $\tilde{x}$, then $w(t')>0$, and therefore
$$ \frac{A_{Y,\Delta_Y}(w)}{S_{Y,\Delta_Y}(\mathrm{wt}_{(1,2)};w)}>\frac{2}{3}.$$
By Lemma \ref{HLQ}, we obtain that no $x$-centered valuation $q$ satisfies
$$ \frac{A_{X,\Delta}(q)}{S_{X,\Delta}(\nu_0;q)}\le \frac{2}{3}.$$

\smallskip

Since $X$ is a regular surface, every non-closed center of a valuation is a prime curve $D$ on $X$, and moreover, every valuation centered on $D$ is a form of $\lambda\cdot \mathrm{ord}_D$ for $\lambda>0$.

\smallskip

Let $D=(h=0)$, and let $a:=\nu_0(h)$. then,
$$ S_{X,\Delta}(\nu_0;\mathrm{ord}_D)=\frac{A_{X,\Delta}(\nu_0)}{2\nu_0(h)}=\frac{3}{4a}.$$
Indeed,
$$ \mathcal{F}^q_{\mathrm{ord}_D}(R/\mathfrak{a}_m(\nu_0))=\frac{(h^q)+\mathfrak{a}_m(\nu_0)}{\mathfrak{a}_m(\nu_0)},$$
and
$$ (\mathfrak{a}_m(\nu_0):h^q)=\mathfrak{a}_{m-qa}(\nu_0).$$
Hence,
$$ \mathrm{vol}\left(\nu_0;\frac{\mathrm{ord}_D}{t}\right)=\mathrm{vol}(\nu_0)(1-at)^2_+,$$
and
$$ S_{X,\Delta}(\nu_0;\mathrm{ord}_D)=\frac{3}{2}\frac{A_{X,\Delta}(\nu_0)}{\mathrm{vol}(\nu_0)}\int^{\infty}_0\mathrm{vol}(\nu_0)(1-at)^2_+ \,\mathrm{d}t=\frac{A_{X,\Delta}(\nu_0)}{2a}.$$
\smallskip

\noindent\textbf{Case 1.} $D=C$.

\smallskip

In this case, $h=f$, and $\frac{A_{X,\Delta}(\mathrm{ord}_C)}{S_{X,\Delta}(\nu_0;\mathrm{ord}_C)}=2$.

\smallskip

\noindent\textbf{Case 2.} $D\ne C$.

\smallskip

In this case, $A_{X,\Delta}(\mathrm{ord}_D)=1$. Moreover, since $h\in \mathfrak{m}_x$, $\nu_0(h)\ge \nu_0(\mathfrak{m}_x)\ge 1$. Hence,
$$ \frac{A_{X,\Delta}(\mathrm{ord}_D)}{S_{X,\Delta}(\nu_0;\mathrm{ord}_D)}=\frac{4\nu_0(h)}{3}\ge \frac{4}{3}.$$
Hence, every ratio $A/S$ of a nonclosed centered valuation is larger than $\frac{2}{3}$.

\smallskip

Therefore, $\delta(X,\Delta,\nu_0)=\frac{2}{3}$, and no one valuation $q$ over $X$ satisfies $\frac{A_{X,\Delta}(q)}{S_{X,\Delta}(\nu_0;q)}=\frac{2}{3}$.
\end{example}

In this situation, there is an étale cover $(Y,\Delta_Y)\to (X,\Delta)$ for which $\delta(Y,\Delta_Y,\nu_0)$ attains its minimum. This leads us to the following question.

\begin{question}
Let $x\in (X,\Delta)$ be the henselization of a klt pair of a local scheme essentially of finite type over an algebraically closed field of characteristic $0$, and $\nu_0$ a non-zero valuation over $X$ whose center is $x$. Can we find any minimizer of $\delta(X,\Delta;\nu_0)$?
\end{question}

A main motivation for finding a minimizer of an invariant is to prove a kind of \emph{stable degeneration theorem} (cf. \cite{Blu18, Xu20, XZ20, XZ25}). Hence, it is natural to ask whether any minimizer has the finitely generated Rees algebra.

\begin{question}
Let $x\in (X:=\Spec R,\Delta,\mathbb{T})$ be a log Fano cone singularity, and let $\nu_0\in \mathrm{Val}^{\mathbb{T},*}_{X,x}$ be a $\mathbb{T}$-invariant valuation with $A_{X,\Delta}(\nu_0)<\infty$. Let $\nu\in \mathrm{Val}^{\mathbb{T},*}_{X,\ni x}$ be a computing valuation of $\delta(X,\Delta;\nu_0)$.

\smallskip

Can we say that $\mathrm{gr}_{\nu}R$ is a finitely generated $k$-algebra?
\end{question}

The answer is \emph{yes} in the case where $\nu_0=\mathrm{wt}_{\xi}$ for some Reeb vector $\xi\in \mathfrak{t}^+_{\R}$, or when $\nu$ is the minimizer of the normalized local volume $\widehat{\mathrm{vol}}(x,X,\Delta)$ (see \cite{Hua22, XZ25}). We also remark that related questions appear in \cite[Question 5.3]{BLXZ23} and \cite[Theorem 6.4]{HMQWZ26}.

\bibliographystyle{halpha}
\bibliography{biblio}

\end{document}